\documentclass{article}
\usepackage[margin=1.2in]{geometry} 
\usepackage{amsmath,amsthm,amssymb,amsfonts}
\usepackage{xcolor}
\usepackage{mathrsfs}
\usepackage{palatino}
\usepackage{verbatim}
\usepackage{float}
\usepackage{graphicx}
\usepackage{tikz-cd}
\usepackage{marvosym}
\usepackage{hyperref}
\usepackage{ytableau}
\usepackage{multicol}
\usepackage{dynkin-diagrams}

\newcommand{\Z}{\mathbb{Z}}
\newcommand{\R}{\mathbb{R}}
\newcommand{\Q}{\mathbb{Q}}
\newcommand{\C}{\mathbb{C}}
\newcommand{\HH}{\mathbb{H}}

\newcommand{\CP}{\mathbb{CP}}
\newcommand{\RP}{\mathbb{RP}}

\DeclareMathOperator{\Id}{Id}

\DeclareMathOperator{\OO}{O}

\DeclareMathOperator{\GL}{GL}
\DeclareMathOperator{\SO}{SO}

\DeclareMathOperator{\RRef}{Ref}
\DeclareMathOperator{\defect}{def}
\DeclareMathOperator{\Bl}{Bl}
\DeclareMathOperator{\Fix}{Fix}
\DeclareMathOperator{\Pic}{Pic}

\DeclareMathOperator{\Aut}{Aut}
\DeclareMathOperator{\Mod}{Mod}
\DeclareMathOperator{\Diff}{Diff}
\DeclareMathOperator{\Homeo}{Homeo}

\DeclareMathOperator{\pr}{pr}
\DeclareMathOperator{\Cr}{Cr}

\newtheorem{thm}{Theorem}[section]
\newtheorem{cor}[thm]{Corollary}
\newtheorem{prop}[thm]{Proposition}
\newtheorem{lem}[thm]{Lemma}

{
\theoremstyle{definition}
\newtheorem{defn}[thm]{Definition}
}
{
\theoremstyle{remark}
\newtheorem{rmk}[thm]{Remark}
}

\begin{document}
\title{Isotopy classes of involutions of del Pezzo surfaces}
\author{Seraphina Eun Bi Lee}
\date{}
\maketitle
\begin{abstract}
Let $M_n := \CP^2 \# n\overline{\CP^2}$ for $0 \leq n \leq 8$ be the underlying smooth manifold of a degree $9-n$ del Pezzo surface. We prove three results about the mapping class group $\Mod(M_n) := \pi_0(\Homeo^+(M_n))$:
\begin{enumerate}
	\item the classification of, and a structure theorem for, all involutions in $\Mod(M_n)$,
	\item a positive solution to the smooth Nielsen realization problem for involutions of $M_n$, and
	\item a purely topological characterization of three remarkable types of involutions on certain $M_n$ coming from birational geometry: de Jonqui\'eres involutions, Geiser involutions, and Bertini involutions.
\end{enumerate}
One main ingredient is the theory of hyperbolic reflection groups.
\end{abstract}

\section{Introduction}\label{sec:introduction}

A \emph{del Pezzo surface} is a smooth projective algebraic surface with ample anticanonical divisor class. Any del Pezzo surface is isomorphic to $\CP^1 \times \CP^1$, $\CP^2$, or $\Bl_P\CP^2$ where $P$ is a set of $n$ points (with $1 \leq n \leq 8$) in general position (no three collinear points, no six coconic points, and no eight points on a cubic which is singular at any of the eight points); see \cite[Proposition 8.1.25]{dolgachev}. The degree of the del Pezzo surface $\Bl_P\CP^2$ is $9 - \lvert P \rvert$, the degree of $\CP^1 \times \CP^1$ is $8$, and the degree of $\CP^2$ is $9$.

The smooth $4$-manifolds underlying del Pezzo surfaces are well-understood; we call such manifolds \emph{del Pezzo manifolds}. The blowup of $\CP^2$ at a finite set $P$ of $n$ points is diffeomorphic to the smooth $4$-manifold
\[
	M_n := \CP^2 \# n \overline{\CP^2}.
\]
In particular, the smooth $4$-manifold underlying a del Pezzo surface of degree $1 \leq d \leq 9$ is $M_{9-d}$ if $d \neq 8$ and $M_1$ or $M_* := \CP^1\times \CP^1$ if $d = 8$. Therefore, the manifolds $M_n$ for $0 \leq n \leq 8$ and $M_*$ make up the list of all del Pezzo manifolds. 

In this paper we relate a property (which we call \emph{irreducibility}) of elements of the mapping class group $\Mod(M) := \pi_0(\Homeo^+(M))$ for all del Pezzo manifolds $M$ to the classification of conjugacy classes of order $2$ elements of the group of birational automorphisms of $\CP^2$. In doing so, we realize all order $2$ mapping classes of del Pezzo manifolds by order $2$ diffeomorphisms coming from a construction that we call \emph{complex equivariant connected sums}. This yields an affirmative solution to the smooth Nielsen realization problem for involutions of del Pezzo manifolds, which is different from the solution for some other $4$-manifolds; for example, Farb--Looijenga (\cite{farb--looijenga}) study the Nielsen realization problem for K3 surfaces and show that not all order $2$ mapping classes of K3 surfaces can be smoothly realized by involutions (or even by diffeomorphisms of finite order). See Remark \ref{rmk:farb--looijenga} below. 

\bigskip
\noindent
{\bf{Irreducibility of mapping classes.}}
Let $M$ be a del Pezzo manifold and let $Q_M$ be the intersection form for $M$. If there exist $(A_1, Q_1)$ and $(A_2, Q_2)$ where $A_i$ is a free $\Z$-module and $Q_i$ is a symmetric bilinear form on $A_i$ with an isometry 
\[
	\iota: (A_1 \oplus A_2, Q_1 \oplus Q_2) \to (H_2(M; \Z), Q_M)
\]
then there exists a natural induced inclusion 
\[
	\Aut(A_1, Q_1) \times \Aut(A_2, Q_2) \hookrightarrow \Aut(H_2(M; \Z), Q_M).
\]
By theorems of Freedman (\cite{freedman}) and Quinn (\cite{quinn}), there is an isomorphism $\Phi: \Mod(N)\to \Aut(H_2(N; \Z), Q_N)$ given by $\Phi([f]) = f_*$ for any closed, oriented, and simply connected $4$-manifold $N$. Hence if $(A_i, Q_i)$ for $i = 1, 2$ is of the form $(H_2(N_i, \Z), Q_{N_i})$ for such $4$-manifolds $N_i$, there also exists a natural induced inclusion
\[
	\iota_*: \Mod(N_1) \times \Mod(N_2) \hookrightarrow \Mod(M).
\]
\begin{defn}[{\bf{Irreducibility}}]\label{defn:irreducibility}
Let $M$ be a del Pezzo manifold and let $g \in \Mod(M)$. Suppose there exist a del Pezzo manifold $N$ and some $k \in \Z_{> 0}$ such that there is an isometry
\[
	\iota: (H_2(N; \Z) \oplus H_2(\#k\overline{\CP^2}; \Z), Q_N \oplus Q_{\#k\overline{\CP^2}}) \to (H_2(M; \Z), Q_{M})
\]
and $g$ is contained in the image of $\iota_*$. Then $g$ is called \emph{reducible}. Otherwise, $g$ is called \emph{irreducible}. 
\end{defn} 
Equivalently, $g$ is reducible if there is some isometry $\iota$ as given above such that under this isometry, $g$ preserves $H_2(N; \Z)$ and $H_2(\#k\overline{\CP^2};\Z)$ when considered as an automorphism of $H_2(M; \Z)$. The restriction of $g$ to $H_2(\#k\overline{\CP^2}; \Z)$ acts by an element of the finite group $\OO(k)(\Z) := \OO(k) \cap \GL(k,\Z)$. 

\bigskip
\noindent
{\bf{Involutions in the plane Cremona group.}}
On the other hand, we consider the mapping classes of automorphisms of complex surfaces induced by involutions in the plane Cremona group. It is known that there are three types of order $2$ conjugacy classes in the group of birational automorphisms of $\CP^2$; they are represented by de Jonqui\'eres involutions, Geiser involutions, and Bertini involutions. This classification was first given by Bertini in 1877 (\cite{bertini}) and proven later by Bayle--Beauville (\cite{bayle--beauville}). The Geiser and Bertini involutions lift to complex automorphisms of del Pezzo surfaces of degree $2$ and $1$ respectively. The de Jonqui\'eres involutions lift to complex automorphisms of blowups of $\CP^2$ at finitely many points; because these points are not necessarily in general position, de Jonqui\'eres involutions do not generally lift to automorphisms of del Pezzo surfaces. We prove in Subsection \ref{sec:irreducible-involutions} that the mapping classes of these involutions as diffeomorphisms of del Pezzo manifolds are irreducible.

\bigskip
\noindent
{\bf{Main results.}}
Throughout, we say that an order $2$ mapping class $g \in \Mod(M)$ of a del Pezzo manifold $M$ is \emph{realized} by an automorphism or anti-biholomorphism $f$ of some complex surface $X$ if $f$ has order $2$ and if there exists a diffeomorphism $\varphi: M \to X$ so that $[\varphi^{-1} \circ f \circ \varphi] = g$. This is a special case of a finite subgroup of $\Mod(M)$ realized by a \emph{complex equivariant connected sum}; see Definition \ref{defn:cecs}(\ref{defn:cecs-1}) and the subsequent remarks.

Our main result is a classification of irreducible mapping classes of order $2$ of del Pezzo manifolds. 
\begin{thm}[{\bf{Characterizing de Jonqui\'eres-Geiser-Bertini}}]\label{thm:involutions}
All mapping classes of $\Mod(M_0)$ and $\Mod(M_*)$ are irreducible. For $1 \leq n \leq 8$, an order two element $g \in \Mod(M_n)$ is irreducible if and only if there exists a complex surface $\Bl_P\CP^2$ with $\lvert P \rvert = n$ such that
\begin{enumerate}
\item $g$ is realized by a complex automorphism of $X = \Bl_P\CP^2$ induced by a de Jonqui\'eres involution of (algebraic) degree $d > 2$, a Geiser involution, or a Bertini involution,\footnote{See Section \ref{sec:irreducible-involutions} for definitions of these involutions.} where $P$ is the set of its base points, or
\item $g$ is realized by an order $2$ anti-biholomorphism given by a composition $f \circ \tau$, where $\tau$ is an order $2$ anti-biholomorphism of $\Bl_P\CP^2$ induced by complex conjugation on $\CP^2$ and $f$ is an automorphism of $\Bl_P\CP^2$ induced by a de Jonqui\'eres involution of (algebraic) degree $d > 2$, a Geiser involution, or a Bertini involution, where $P$ is the set of its base points.
\end{enumerate}
\end{thm}
There is an index $2$ subgroup $\Mod^+(M_n)$ of $\Mod(M_n)$ for which the following simpler version of Theorem \ref{thm:involutions} holds; see Definition \ref{defn:mod-plus} for a precise description of $\Mod^+(M_n)$.
\begin{thm}[{\bf{Irreducibility classification for $\Mod^+(M_n)$}}]\label{thm:involutions-pt1}
For $1 \leq n \leq 8$, an order two element $g \in \Mod^+(M_n)$ is irreducible if and only if $g$ is realized by a complex automorphism of a complex surface $X = \Bl_P\CP^2$ induced by a de Jonqui\'eres involution of (algebraic) degree $d > 2$, a Geiser involution, or a Bertini involution, where $P$ is the set of its base points.
\end{thm}
Using the theory of hyperbolic reflection groups and Carter's classification of conjugacy classes of Weyl groups (\cite{carter}), we enumerate the conjugacy classes of involutions in $\Mod^+(M_n)$, of which we study the irreducible ones to prove Theorem \ref{thm:involutions-pt1}. We then extend Theorem \ref{thm:involutions-pt1} to Theorem \ref{thm:involutions} by exhibiting some birational involutions of $\CP^2$ that commute with complex conjugation.

In their classification of conjugacy classes of order $2$ elements of $\Cr(2)$, Bayle--Beauville (\cite{bayle--beauville}) study pairs $(S, \sigma)$ where $S$ is a rational surface and $\sigma \in \Aut(S)$ has order $2$. Such a pair is called \emph{minimal} if $f: S \to S'$ is a birational morphism such that there exists an involution $\sigma' \in \Aut(S')$ and $f\circ \sigma = \sigma' \circ f$ then $f$ is an isomorphism. Bayle--Beauville classify all minimal pairs $(S, \sigma)$ (\cite[Theorem 1.4 and Proposition 1.7]{bayle--beauville}); applying this classification yields the following simple reformulation of Theorem \ref{thm:involutions-pt1}:
\begin{cor}[{\bf{Minimal pairs}}]\label{cor:minimal}
Let $M$ be a del Pezzo manifold. An order $2$ mapping class $g \in \Mod^+(M)$ is irreducible if and only if it is realized by a minimal pair $(S, \sigma)$ where $S$ is diffeomorphic to $M$.
\end{cor}

Up to conjugacy, every mapping class of a del Pezzo manifold $M$ is specified by an irreducible mapping class of some del Pezzo manifold and an involution in $\OO(k)(\Z)$ acting on $H_2(\#k\overline{\CP^2}; \Z)$ for some $k \geq 0$. Theorem \ref{thm:involutions} shows that mapping classes of order $2$ are built out de Jonqui\'eres, Geiser, and Bertini involutions and involutions of $M_0$ and $M_*$. In Section \ref{sec:cecs} we describe the construction of \emph{complex equivariant connected sums} that builds smooth involutions representing reducible mapping classes out of biholomorphisms or anti-biholomorphisms of order $2$ that represent irreducible mapping classes. See Figure \ref{fig:type2} for an example of an equivariant connected sum. The smooth Nielsen realization problem for involutions then follows from Theorem \ref{thm:involutions}.
\begin{cor}[{\bf{Nielsen realization for involutions}}]\label{cor:involutions}
Let $M$ be a del Pezzo manifold. Any order $2$ element $g \in \Mod(M)$ is realized by a smooth involution. In fact, $g$ is realized by a complex equivariant connected sum. 
\end{cor}
\begin{figure}
\centering
\includegraphics{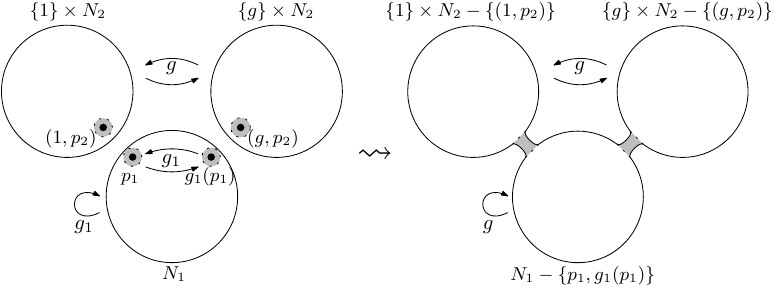}
\caption{The equivariant connected sum $(N_1\#(G\times N_2), G)$ where $G = \langle g \rangle \cong \Z/2\Z$. If $N_1$ is a del Pezzo manifold and $N_2 \cong \overline{\CP^2}$, the action of $G$ on $N_1 \#(G\times N_2)$ induces a reducible mapping class.}\label{fig:type2}
\end{figure}
\begin{rmk}\label{rmk:lee}
The main results of \cite{lee} show that finite subgroups $G \leq \Mod(M_2)$ and maximal finite subgroups $G \leq \Mod(M_3)$ have lifts to $\Diff^+(M_2)$ and $\Diff^+(M_3)$ respectively under the map $\pi: \Diff^+(M_n) \to \Mod(M_n)$ if and only if they are realized by a complex equivariant connected sum. Corollary \ref{cor:involutions} is an analogous statement for the case $G \cong \Z/2\Z$ and any del Pezzo manifold $M$.

The construction of complex equivariant connected sums is necessary in the solution for the smooth Nielsen realization problem. For all $n \geq 1$, there exist mapping classes $g \in \Mod(M_n)$ of order $2$ that cannot be realized by complex automorphisms of any complex structure on $M_n$ by \cite[Theorem 1.8]{lee} even though they can be realized by complex equivariant connected sums.
\end{rmk}
\begin{rmk}\label{rmk:farb--looijenga}
A special case of Corollary \ref{cor:involutions} says that for any Dehn twist $T$ about a $(-2)$-sphere in any del Pezzo manifold $M$, there is an order $2$ diffeomorphism of $M$ (topologically) isotopic to $T$. (For the case $M = M_2$, this is the statement of \cite[Corollary 1.3]{lee}.) In contrast, Farb--Looijenga (\cite[Corollary 1.10]{farb--looijenga}) shows that the (topological) isotopy class of any Dehn twist about a $(-2)$-sphere in a K3 surface is not represented by any finite order diffeomorphism. 
\end{rmk}

\bigskip
\noindent
{\bf{Related work.}} 
This paper is a followup to \cite{lee}. As described in Remark \ref{rmk:lee}, we examine a similar phenomenon in \cite{lee} in which finite subgroups of the mapping class groups of del Pezzo manifolds of high degree are realized by diffeomorphisms if and only if they are realized by complex equivariant connected sums. 

As noted above, Bayle--Beauville (\cite{bayle--beauville}) prove the classification of order $2$ conjugacy classes of the plane Cremona group. Their proof involves studying minimal pairs $(S, f)$ where $S$ is a rational surface and $f \in \Aut(S)$ is an involution. We only invoke the classification (of minimal pairs or order $2$ conjugacy classes in $\Cr(2)$) of Bayle--Beauville in the proof of Corollary \ref{cor:minimal}.

Hambleton--Tanase (\cite[Theorem A]{hambleton--tanase}) show that if $G = \Z/p\Z$ acts smoothly on $\#n \CP^2$ for $n \geq 1$ and $p$ is an odd prime then there exists an equivariant connected sum of linear actions on $\CP^2$ with the same fixed-set data (see \cite{hambleton--tanase} for the exact description of this data) and the same induced action on $H_2(\#n\CP^2; \Z)$. Corollary \ref{cor:involutions} of our paper is similar in flavor in that all involutions on $H_2(M; \Z)$ for del Pezzo manifolds arise from a complex equivariant connected sum. However, our methods are much more elementary than those of Hambleton--Tanase (\cite{hambleton--tanase}) who utilize the theory of equivariant Yang--Mills moduli spaces; conversely, our methods do not yield as much information about the fixed sets of such involutions.

For some other examples of $4$-manifolds, the existence of order $2$ mapping classes of $4$-manifolds that do not lift to an order $2$ diffeomorphism was known; see Raymond--Scott (\cite[Theorem 1]{raymond--scott}) for the case of certain nil-manifolds (in every dimension $d \geq 3$) and Baraglia--Konno (\cite[Theorem 1.2]{baraglia--konno}) for the case of the K3 manifold. The Nielsen realization problem for $4$-manifolds was first studied by Farb--Looijenga in their recent paper \cite{farb--looijenga}. Specifically, Farb--Looijenga study the case of K3 surfaces and solve the metric and complex Nielsen realization problem for all finite groups as well as the smooth Nielsen realization problem for $\Z/2\Z$. Their results show, in particular, that Dehn twists in the K3 manifold are not realized by finite-order diffeomorphisms (\cite[Corollary 1.10]{farb--looijenga}); this result was later extended to all smooth spin $4$-manifolds with non-zero signature by Konno (\cite[Theorem 1.1]{konno}).

\bigskip
\noindent
{\bf{Outline of this paper.}} 
In Section \ref{sec:ModM} we outline the tools necessary to enumerate and study involutions in $\Mod(M)$ and to realize these mapping classes in $\Diff^+(M)$. Section \ref{sec:involutions} is dedicated to the proof of Theorem \ref{thm:involutions}. More specifically, Sections \ref{sec:n04} and \ref{sec:n58} analyze involutions contained in some index $2$ subgroup $\Mod^+(M) \leq \Mod(M)$ for each del Pezzo manifold $M$. Section \ref{sec:irreducible-involutions} describes and examines the three types of conjugacy classes of involutions in the plane Cremona group. Finally, Section \ref{sec:extension} extends the result for $\Mod^+(M)$ to $\Mod(M)$. Finally, Section \ref{sec:nielsen} contains the proof of Corollary \ref{cor:involutions}.

\bigskip
\noindent
{\bf{Acknowledgements.}}
I am grateful to Benson Farb for his support and guidance on this project which been truly invaluable. I thank Farb and Eduard Looijenga for sharing an earlier draft of their paper \cite{farb--looijenga} which shaped and inspired this project. I thank Hokuto Konno for bringing many relevant references on Nielsen realization for $4$-manifolds to my attention. I also thank Danny Calegari and Shmuel Weinberger for their answers to my questions about mapping class groups of $4$-manifolds and finite group actions on $4$-manifolds, Ishan Banerjee for insightful conversations, and R. \.{I}nan\c{c} Baykur, Anubhav Mukherjee, and Nick Salter for their comments on an earlier draft of this paper. Finally, I thank the anonymous referee for their careful reading and for offering many valuable comments and suggestions.

\section{Mapping class groups of del Pezzo manifolds}\label{sec:ModM}
In this section we outline some tools used to study the mapping class groups of del Pezzo manifolds in this paper.
\subsection{The mapping class group}
The Mayer--Vietoris sequence implies that $H_2(M_n; \Z) = H_2(\CP^2; \Z) \oplus H_2(\overline{\CP^2}; \Z)^{\oplus n}$ for any $0 \leq n \leq 8$ and gives a natural $\Z$-basis $\{H, E_1, \dots, E_n\}$ with intersection form $Q_{M_n} \cong \langle 1 \rangle \oplus n \langle -1 \rangle$. The group $\Aut(H_2(M; \Z), Q_{M_n})$ is the indefinite orthogonal group $\OO(1,n)(\Z)$, i.e. by theorems of Freedman (\cite{freedman}) and Quinn (\cite{quinn}),
\[
	\Mod(M_n)\cong \OO(1,n)(\Z).
\]
Next, consider $M = M_* := \CP^1 \times \CP^1$. The lattice $(H_2(\CP^1 \times \CP^1; \Z), Q_{\CP^1 \times \CP^1})$ has two isotropic generators $S_1$ and $S_2$ with $Q_{\CP^1 \times \CP^1}(S_1, S_2) = 1$ coming from the factors of the product $\CP^1 \times \CP^1$. We will identify $\Aut(H_2(M; \Z), Q_M)$ and $\Mod(M)$ for all del Pezzo manifolds $M$ in this paper.

Let $0 \leq k < n$ and let $v_1, \dots, v_{(n-k)}$ denote the orthogonal $\Z$-basis of $H_2(\#(n-k)\overline{\CP^2}; \Z)$. There is an isometry
\[
	\iota_k: (H_2(M_k; \Z), Q_{M_k}) \oplus (H_2(\#(n-k)\overline{\CP^2}; \Z), Q_{\#(n-k)\overline{\CP^2}}) \to (H_2(M_n; \Z), Q_{M_n})
\]
such that for $1 \leq i \leq k$ and $1 \leq j \leq n-k$,
\[
	\iota_k((H, 0)) =  H, \qquad \iota_k((E_i, 0)) = E_i, \qquad \iota_k((0, v_j))= E_{k+j}.
\]
Moreover, there is an isometry
\[
	\iota: (H_2(M_*; \Z), Q_{M_*}) \oplus (H_2(\#(n-1)\overline{\CP^2}; \Z), Q_{\#(n-1)\overline{\CP^2}}) \to (H_2(M_n; \Z), Q_{M_n})
\]
such that for $i = 1, 2$ and $2 \leq j \leq n-1$,
\[
	\iota((S_i, 0)) = H-E_i, \qquad \iota((0, v_1))= H-E_1-E_2, \qquad \iota((0, v_j)) = E_{1+j},
\]
where $S_1$ and $S_2$ denote the two isotropic generators of $H_2(M_*; \Z)$ as above. 
\begin{defn}\label{defn:standard-inclusion}
There exist induced inclusions 
\[
	(\iota_k)_*: \Mod(M_k) \times \Mod(\#(n-k)\overline{\CP^2}) \hookrightarrow \Mod(M_n)
\] for $0 \leq k < n$ and 
\[
	\iota_*: \Mod(M_*) \times \Mod(\#(n-1)\overline{\CP^2}) \hookrightarrow \Mod(M_n)
\]
by theorems of Freedman (\cite{freedman}) and Quinn (\cite{quinn}); see the discussion preceding Definition \ref{defn:irreducibility}. The inclusions $(\iota_k)_*$ and $\iota_*$ are called \emph{standard inclusions}. 
\end{defn}

Note that for $n \geq 2$, $M_n$ is diffeomorphic to $(\CP^1\times\CP^1)\#(n-1)\overline{\CP^2}$. Applying \cite[Theorem 2]{wall-diffeos} to $M_n$ with this diffeomorphism yields the following statement. (The same statement holds for $M_0$, $M_*$, and $M_1$; for example, see Lemma \ref{lem:nielsen-base}.)
\begin{thm}[{Wall, \cite[Theorem 2]{wall-diffeos}}]\label{thm:diffeo-realizable}
For $M = M_*$ or $M_n$ with $2 \leq n \leq 9$, the restriction of $\pi: \Homeo^+(M) \to \Mod(M)$ to the subgroup $\Diff^+(M) \leq \Homeo^+(M)$ is surjective.
\end{thm}
\begin{rmk}
Theorem \ref{thm:diffeo-realizable} cannot be extended to manifolds $M_n$ for $n \geq 10$; Friedman--Morgan (\cite[Theorem 10]{friedman--morgan}) show that the image of the quotient $\pi|_{\Diff^+(M_n)}: \Diff^+(M_n) \to \Aut(H_2(M_n; \Z), Q_{M_n})$ has infinite index in $\Aut(H_2(M_n; \Z), Q_{M_n})$ for all $n \geq 10$.
\end{rmk}

\subsection{Coxeter theory and the group $\OO^+(n,1)(\Z)$}\label{sec:coxeter}
Fix $2 \leq n \leq 9$ and consider the symmetric, bilinear form on $\R^{n+1}$ defined by
\[
	R_n((x_0, x_1, \dots, x_n), (y_0, y_1, \dots, y_n)) = -x_0 y_0 + x_1 y_1 + \dots + x_n y_n.
\] 
We identify $\R^{n+1}$ with $H_2(M_n; \Z) \otimes \R$ such that the ordered $\Z$-basis $\{H, E_1, \dots, E_n\}$ is identified the given ordered basis of $\R^{n+1}$. Then $R_n$ is precisely the bilinear form $-Q_{M_n}$ extended $\R$-linearly. For any $v \in \Z^{n+1} \subseteq \R^{n+1}$ with $R_n(v,v) = \pm 1$, $\pm 2$, a \emph{reflection} $\RRef_v$ about $v$ defines an involution in $\OO(n,1)(\Z)$ by
\[
	\RRef_v(w) := w - \frac{2 R_n(v, w)}{R_n(v,v)} v.
\]

For any $n \geq 0$, let $\OO^+(n,1)(\Z)$ be the index $2$ subgroup of $\OO(n,1)(\Z)$ defined
\[
	\OO^+(n,1)(\Z) := \{A \in \OO(n,1)(\Z) : A(H) = aH + b_1 E_1 + \dots + b_n E_n, \, a > 0 \}.
\]
Wall gives explicit generators of $\OO^+(n,1)(\Z)$ for $2 \leq n \leq 9$ in terms of reflections:
\begin{thm}[{Wall, \cite[Theorems 1.5, 1.6]{wall-indefinite-orthogonal}}]\label{thm:reflections-wall}
For $n = 2$,
\[
	\OO^+(2,1)(\Z) = \langle \RRef_{H+E_1+E_2}, \, \RRef_{E_1-E_2}, \, \RRef_{E_2} \rangle.
\]
For $3 \leq n \leq 9$,
\[
	\OO^+(n,1)(\Z) = \langle \RRef_{H+E_1+E_2+E_3},\, \RRef_{E_1-E_2},\, \RRef_{E_2-E_3},\, \dots, \, \RRef_{E_{n-1}-E_n}, \, \RRef_{E_n}  \rangle.
\]
\end{thm}
\begin{rmk}
Another way to phrase the first equality of Theorem \ref{thm:reflections-wall} is that $\OO^+(2,1)(\Z)$ is the triangle group $\Delta(2, 4, \infty)$. This formulation is classical, shown by Fricke in \cite[p. 64-68]{fricke}.
\end{rmk}
It is straightforward to show that $\OO^+(n,1)(\Z)$ is the Coxeter group corresponding to the Coxeter system $(W(n), S(n))$, where
\[
	S(n) := \begin{cases}
	\{\RRef_{H-E_1-E_2}, \, \RRef_{E_1-E_2}, \, \RRef_{E_2}\} & \text{if }n = 2, \\
	\{\RRef_{H-E_1-E_2-E_3}, \, \RRef_{E_1-E_2}, \, \RRef_{E_2-E_3}, \, \dots, \, \RRef_{E_{n-1}-E_n}, \, \RRef_{E_n}\} & \text{if }3 \leq n \leq 9.
	\end{cases}
\]
The Coxeter diagrams for $(W(n), S(n))$ with $2 \leq n \leq 9$ are given in Figure \ref{fig:coxeter}.
\begin{figure}
\includegraphics{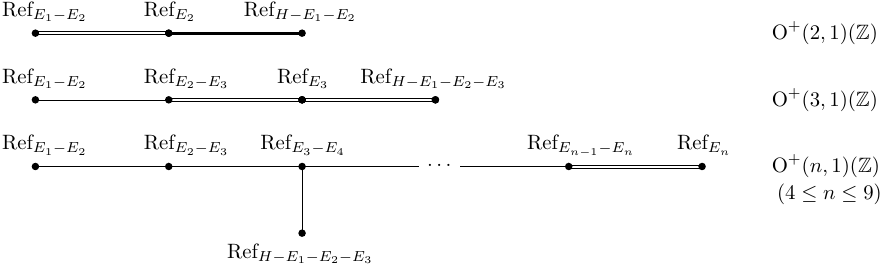}
\caption{The Coxeter diagrams for $\OO^+(n,1)(\Z)$ for $2 \leq n \leq 9$. For fixed $n$, we refer to the specified Coxeter system as $(W(n), S(n))$.} \label{fig:coxeter}
\end{figure}

Let $V_n$ be the $\R$-span of $\{\alpha_s : s \in S(n)\}$ on which $\OO^+(n,1)(\Z)$ acts by the \emph{geometric representation} of $(W(n), S(n))$ and let $B_n$ be the standard symmetric bilinear form of $V_n$ as defined in \cite[Section 5.3]{humphreys}. The signature of $B_n$ is $(n,1)$. There is an isometry $F_n: (V_n, B_n) \to (\R^{n+1}, R_n)$ given on the basis elements of $V_n$ by $F_n(\alpha_{\RRef_v}) = R_n(v,v)^{-\frac 12}v$. One can check that $F_n(s\cdot v) = s \cdot F_n(v)$ for all $v \in V_n$ and $s \in S(n)$. Finally, the submanifold of $\R^{n+1}$ given by
\[
	\HH^n = \{v = (v_0, \dots, v_n) \in \R^{n+1}: v_0 > 0, \, R_n(v,v) = -1\}
\]
with the metric induced by $R_n$ is isometric to hyperbolic $n$-space (see \cite[Chapter 2]{thurston}).

The fact that $\OO^+(n,1)(\Z)$ acts on $\mathbb H^n$ by isometries via the geometric representation of $(W(n), S(n))$ allows for an easy classification of involutions in $\OO^+(n,1)(\Z)$. 
\begin{lem}\label{lem:involution-classification}
Fix $2 \leq n \leq 9$. Suppose $g \in \OO^+(n,1)(\Z)$ has finite order. 
\begin{enumerate}
\item Up to conjugacy in $\OO^+(n,1)(\Z)$, the element $g$ is contained in a subgroup $G_v := \langle s \in S(n) - \{\RRef_v\} \rangle$ for some $\RRef_v \in S(n) - \{\RRef_{E_1-E_2}\}$.
\item Suppose that there does not exist any isometries
\[
	\iota: (H_2(N; \Z), -Q_N) \oplus (\Z^k, k\langle 1 \rangle) \to (\Z^{n+1}, R_n).
\]
where $k > 0$ and $N$ is some del Pezzo manifold such that $g$ preserves the images of each summand under $\iota$. Then $g \in G_{E_n}$ up to conjugacy in $\OO^+(n,1)(\Z)$. 
\end{enumerate}
\end{lem}
\begin{proof}
\begin{enumerate}
	\item The fundamental domain of the action of $\OO^+(n,1)(\Z)$ on $\HH^n \subseteq (\R^{n+1}, R_n)$ is given by 
	\[
		P := \bigcap_{\RRef_v \in S(n)}\{w \in \HH^n : R_n(w,v) \leq 0\}
	\]
	by \cite[Proposition 4, Table 4]{vinberg}, after conjugating the generators $S(n)$ by the element of $\OO^+(n,1)(\Z)$ which negates each $E_i$ and fixes $H$. If $U \subseteq V_n$ denotes the Tits cone of $(W(n), S(n))$ then $F_n^{-1}(P)$ is contained in $-U := \{u \in V_n: -u \in U\}$. Hence $F_n^{-1}(\mathbb H^n)$ is also contained in $-U$. 

	The finite subgroup $\langle g \rangle$ acts on $\HH^n$. The group $\langle g \rangle$ must fix a point $\HH^n$ by \cite[Corollary 2.5.19]{thurston}. Therefore it must fix a point $F_n^{-1}(\mathbb H^n) \subseteq -U$, and hence also a point $p$ in the Tits cone $U$. By \cite[Theorem 5.13]{humphreys}, the stabilizer of $p$ in $\OO^+(n,1)(\Z)$ is
	\[
		W_I := \langle s \in I \subseteq S(n) \rangle
	\]
	for some $I \subseteq S(n)$, up to conjugation in $\OO^+(n,1)(\Z)$. If $I = S(n)$ then the only fixed point of $W_I$ in $V_n$ is $0$, which is not contained in $F_n^{-1}(\HH^n)$. If $I = S(n) - \{\RRef_{E_1-E_2}\}$, the fixed subspace of $W_I$ in $V_n$ is $F_n^{-1}(\R\{H - E_1\})$, which has empty intersection with $F_n^{-1}(\HH^n)$. Therefore, $g \in G_v$ for some $v$ such that $\RRef_v \in S(n) - \{\RRef_{E_1-E_2}\}$. 

	\item By the first part of this lemma, $g \in G_v$ up to conjugacy in $\OO^+(n,1)(\Z)$ for some $v \in S(n)- \{\RRef_{E_1-E_2}\}$. For all decompositions of $\Z^{n+1}$ given below, the restriction of $R_n$ to the last summand is diagonal and positive definite.
	\begin{enumerate}
		\item If $v = E_2-E_3$ then $G_v$ preserves the summands in the decomposition
		\[
			\Z^{n+1} = \Z\{H-E_1,\, H-E_2\} \oplus \Z\{H-E_1-E_2,\, E_3,\, \dots,\, E_n\}.
		\]
		Note $(\Z\{H-E_1,\, H-E_2\}, R_n) \cong (H_2(M_*; \Z), -Q_{M_*})$.
		\item If $v = E_3-E_4$ then $G_v$ preserves the summands in the decomposition
		\[
			\Z^{n+1} = \Z\{H,\, E_1,\, E_2,\, E_3\} \oplus \Z\{E_4,\, \dots,\, E_n\}.
		\]
		Note $(\Z\{H,\, E_1,\, E_2,\, E_3\}, R_n) \cong (H_2(M_3; \Z), -Q_{M_3})$.
		\item If $v = H-E_1-E_2-E_3$ then $G_v$ preserves the summands in the decomposition
		\[
			\Z^{n+1} =\Z\{H\} \oplus \Z\{E_1,\, \dots,\, E_n\}.
		\]
		Note $(\Z\{H\}, R_n) \cong (H_2(M_0; \Z), -Q_{M_0})$.
		\item If $n = 2$ and $v = H-E_1-E_2$ then $G_v$ preserves the summands in the decomposition
		\[
			\Z^{3} = \Z\{H\} \oplus \Z\{E_1, E_2\}.
		\]
		Note $(\Z\{H\}, R_2) \cong (H_2(M_0; \Z), -Q_{M_0})$. 
		\item If $v = E_k - E_{k+1}$ with $k \geq 4$ then $G_v$ preserves the summands in the decomposition
		\[
			\Z^{n+1} = \Z\{H,\, E_1,\, \dots,\, E_k\} \oplus \Z\{E_{k+1},\, \dots,\, E_n\}.
		\]
		Note $(\Z\{H,\, E_1,\, \dots, \, E_k\}, R_n) \cong (H_2(M_k; \Z), -Q_{M_k})$.
	\end{enumerate}
	All subgroups $G_v$ with $v \neq E_n$ preserve some orthogonal decomposition of $(\Z^{n+1}, R_n)$ specified in the statement of the lemma. Therefore $g$ must be contained in $G_{E_n}$. \qedhere
\end{enumerate}
\end{proof}

In the rest of the paper, we often consider the image of the subgroup $\OO^+(1,n)(\Z) \leq \OO(1,n)(\Z)$ in $\Mod(M_n)$ under the isomorphism $\Phi: \OO(1,n)(\Z) \to \Mod(M_n)$.
\begin{defn}\label{defn:mod-plus}
For any $0 \leq n \leq 9$, let $\Mod^+(M_n)$ denote the index $2$ subgroup $\OO^+(1,n)(\Z)$ of $\Mod(M_n)$ under the isomorphism $\Phi:\OO(1,n)(\Z)\to \Mod(M_n)$. Let $\Mod^+(M_*)$ denote the index $2$ subgroup $\langle c \rangle \cong \Z/2\Z$ of $\Mod(M_*)$ under the isomorphism $\Aut(H_2(M_*; \Z), Q_{M_*}) \to \Mod(M_*)$, where $c$ is the map swapping the isotropic generators $S_1$ and $S_2$ of $H_2(M_*; \Z)$.
\end{defn}
With this definition in hand, we reformulate Lemma \ref{lem:involution-classification} as a statement about irreducibility of mapping classes. 
\begin{cor}\label{cor:involution-classification}
Let $2 \leq n \leq 9$ and let $\Phi$ denote the isomorphism $\OO(1,n)(\Z) \to \Mod(M_n)$. Suppose $g \in \Mod^+(M_n)$ has finite order. If $g$ is irreducible then $g$ is in $W_n := \Phi(G_{E_n}) \leq \Mod^+(M_n)$ up to conjugacy in $\Mod^+(M_n)$. 
\end{cor}
\begin{proof}
There is an equality of subgroups $\OO(1,n)(\Z) = \OO(n,1)(\Z) \leq \GL(n+1, \Z)$ and an isomorphism $\Phi: \OO(1,n)(\Z) = \OO(n,1)(\Z) \to \Mod(M_n)$. If $g \in \Mod^+(M_n)$ is irreducible then there does not exist any isometry 
\[
	\iota: (H_2(N; \Z), -Q_N) \oplus (H_2(\#k\overline{\CP^2}; \Z), -Q_{\#k\overline{\CP^2}}) \to (H_2(M_n; \Z), -Q_{M_n}) \cong (\Z^{n+1}, R_n)
\]
such that $\Phi^{-1}(g)$ preserves the image of each summand $(H_2(N; \Z), -Q_N)$ and $(H_2(\#k\overline{\CP^2}; \Z), -Q_{\#k\overline{\CP^2}})$ under $\iota$. Lemma \ref{lem:involution-classification} implies that $\Phi^{-1}(g) \in G_{E_n}$ up to conjugacy in $\Mod^+(M_n)$. 
\end{proof}
For any reflection $\RRef_v \in \OO(1,n)(\Z)$, we also denote the corresponding mapping class $\Phi(\RRef_v)$ by $\RRef_v$ in the rest of the paper. 

\section{Order $2$ elements of $\Mod(M_n)$ with $1 \leq n \leq 8$}\label{sec:involutions}
The goal of this section is to prove Theorems \ref{thm:involutions-pt1} and \ref{thm:involutions} and Corollary \ref{cor:minimal}. 
\subsection{The Weyl group $W(\mathbb E_n)$}\label{sec:En}

Let $X$ be a del Pezzo surface diffeomorphic to $M_n$. By \cite[p. 378]{dolgachev}, the action of any complex automorphism $f \in \Aut(X)$ on $H_2(M_n; \Z)$, denoted by $f_* \in \Aut(H_2(M_n; \Z), Q_{M_n})$, leaves the canonical class $K_X \in H_2(M_n; \Z)$ invariant. The canonical class is given by $K_X = -3H + \sum_{i=1}^n E_i$. 

The restriction of $Q_{M_n}$ to $\mathbb E_n := (\Z\{K_X\})^\perp$ turns $\mathbb E_n$ into an even, negative-definite lattice if $n \leq 8$ by \cite[p. 361]{dolgachev}. For $n \geq 3$, there is a $\Z$-basis of $\mathbb E_n$
\[
	\{H-E_1-E_2-E_3, \, E_1-E_2, \, \dots, \, E_{n-1} -E_n\}
\]
(\cite[Lemma 8.2.6]{dolgachev}). Define the \emph{Weyl group} $W(\mathbb E_n)$ to be the subgroup of $\Mod(M_n)$ generated by the reflections $\RRef_v$ for $v$ in this basis. Observe that $W(\mathbb E_n)$ coincides with the subgroup $W_n$ containing all irreducible involutions of $\Mod^+(M_n)$, up to conjugacy in $\Mod^+(M_n)$, as considered in Corollary \ref{cor:involution-classification}. Moreover, $W_n$ is the stabilizer of $K_X$ in $\OO(1,n)(\Z)$ by \cite[Corollary 8.2.15]{dolgachev} and 
\begin{align*}
W_3 &\cong W(A_2) \times W(A_1), \\
W_4 &\cong W(A_4), \\
W_5 &\cong W(D_5), \\
W_n &\cong W(E_n) \quad\quad \text{for }6 \leq n \leq 8.
\end{align*}
\begin{rmk}
The subgroup of $W_n$ generated by the reflections $\RRef_{E_k-E_{k+1}}$ for $1 \leq k \leq n-1$ is isomorphic to $S_n$ via its action on the set $\{E_1, \dots, E_n\}$.
\end{rmk}

\subsection{Involutions in $\Mod^+(M_n)$ for $n = *$ and $0 \leq n \leq 4$}\label{sec:n04}
In this section we examine the order $2$ elements of $\Mod^+(M)$ for $M = M_n$ with $0 \leq n \leq 4$ and $M = M_*$. We account for the only irreducible mapping classes of order $2$ in $\Mod^+(M)$ for $M = M_*$ or $M_n$ with $0 \leq n \leq 4$ in the following lemma.
\begin{lem}[$n = *$ and $0$]\label{lem:n*0}
Let $M = M_0$ or $M_*$. Any $g \in \Mod(M)$ is irreducible.
\end{lem}
\begin{proof}
There does not exist $c \in H_2(M; \Z)$ such that $Q_M(c,c) = -1$. Therefore, there is no isometric embedding 
\[
	(H_2(\#k\overline{\CP^2}; \Z), Q_{\#k\overline{\CP^2}}) \hookrightarrow (H_2(M; \Z), Q_M)
\]
for any $k > 0$. 
\end{proof}
The rest of the mapping classes of order $2$ considered in this section are reducible.
\begin{lem}[$1 \leq n \leq 4$]\label{lem:n-low}
Let $1 \leq n \leq 4$. If $g \in \Mod^+(M_n)$ has order $2$ then $g$ is reducible.
\end{lem}
\begin{proof}
The group $\Mod^+(M_1)$ is generated by $\RRef_{E_1}$. The group $\Mod^+(\CP^2)$ is trivial and the image of the standard inclusion
\[
	\iota_*: \Mod^+(\CP^2) \times \Mod(\overline{\CP^2}) \hookrightarrow \Mod^+(M_1)
\]
is precisely $\Mod^+(M_1)$. Therefore, any $g \in \Mod^+(M_1)$ is reducible.

The group $W_2 \leq \Mod^+(M_2)$ is generated by $\RRef_{E_1-E_2}$ and $\RRef_{H-E_1-E_2}$ which commute in $\Mod^+(M_2)$. The image of the standard inclusion
\[
	\iota_*: \Mod^+(M_*) \times \Mod(\overline{\CP^2}) \hookrightarrow \Mod^+(M_2)
\]
is precisely $W_2$. Then because any $g \in W_2$ is reducible, Corollary \ref{cor:involution-classification} implies that any $g \in \Mod^+(M_2)$ of finite order is reducible.

The group $W_3 \leq \Mod^+(M_3)$ is given by $\langle \RRef_{H-E_1-E_2-E_3}\rangle \times \langle \RRef_{E_1-E_2}, \RRef_{E_2-E_3} \rangle$. The group $\langle \RRef_{E_1-E_2}, \RRef_{E_2-E_3} \rangle$ is isomorphic to $S_3$ so the elements of order $2$ are conjugate in $S_3$ to $\RRef_{E_1 - E_2}$. Therefore, any $g \in W_3$ of order $2$ is conjugate to $\RRef_{H-E_1-E_2-E_3}\circ\RRef_{E_1-E_2}$ or $\RRef_{E_1-E_2}$ in $W_3$. Replace $g$ with its conjugate $\RRef_{H-E_1-E_2-E_3}\circ\RRef_{E_1-E_2}$ or $\RRef_{E_1-E_2}$ and observe in both cases that $g$ preserves $\Z\{H-E_1,\, H-E_2\}$. Therefore, $g$ is reducible because it is contained in the image of the standard inclusion
\[
	\iota_*:\Mod^+(M_*) \times \Mod(\#2\overline{\CP^2}) \hookrightarrow \Mod^+(M_3).
\]
Because any $g \in W_3$ of order $2$ is reducible, Corollary \ref{cor:involution-classification} implies that any $g \in \Mod^+(M_3)$ of order $2$ is reducible.

By the proof of \cite[Theorem 8.5.8]{dolgachev}, the group $W_4 \leq \Mod^+(M_4)$ is isomorphic to $S_5$ generated by the subgroup $S_4 = \langle \RRef_{E_1-E_2}, \RRef_{E_2-E_3}, \RRef_{E_3-E_4} \rangle$ and an element of order $5$. This means that any $g \in W_4$ of order $2$ is conjugate in $W_4$ to an element in $S_4$. The image of the standard inclusion 
\[
	\iota_*: \Mod^+(\CP^2) \times \Mod(\#4\overline{\CP^2}) \hookrightarrow \Mod^+(M_4)
\]
contains $S_4 \leq W_4$ meaning that $g$ is reducible. Then because any $g \in W_4$ of order $2$ is reducible, Corollary \ref{cor:involution-classification} implies that any $g \in \Mod^+(M_4)$ of order $2$ is reducible. 
\end{proof}

\subsection{Irreducible mapping classes and involutions in the Cremona group} \label{sec:irreducible-involutions}

The smallest integer $n \geq 1$ such that there exist irreducible mapping classes of order $2$ in $\Mod^+(M_n)$ is $n = 5$. In order to discuss these irreducible classes, we first need to consider some classical involutions in the plane Cremona group $\Cr(2)$, i.e. the group of birational automorphisms of $\CP^2$. Conjugacy classes of involutions in the plane Cremona group are classified by the following theorem.
\begin{thm}[{Bayle--Beauville, \cite[Theorem 2.6]{bayle--beauville}}]
Every birational involution of $\CP^2$ is conjugate in $\Cr(2)$ to one and only one of the following:
\begin{enumerate}
	\item a de Jonqui\'eres involution of degree $d \geq 2$,
	\item a Geiser involution, or
	\item a Bertini involution.
\end{enumerate}
\end{thm}
We now briefly recall the definitions of these involutions.
\subsubsection{de Jonqui\'eres involutions}\label{sec:de-jonquieres}
This description of de Jonqui\'eres involutions follows the exposition of \cite[Example 3.1]{blanc}. Fix $g \geq 1$ and $a_1, \dots, a_{2m} \in \C$ distinct with $m = g+1$. Consider the map $\varphi_0: \CP^1 \times \CP^1 \dashrightarrow \CP^1 \times \CP^1$ defined
\[
	\varphi_0: ([X_1:X_2], [Y_1:Y_2]) \mapsto \left( [X_1:X_2], \left[Y_2 \prod_{i=m+1}^{2m} (X_1 - a_i X_2) : Y_1 \prod_{i=1}^m (X_1 - a_i X_2)\right] \right).
\]
The map $\varphi_0$ is rational and defined on the open set $U$, which is the complement of the set of $2m$ points
\[
	P =\{p_i = ([a_i:1],[1:0]) : 1 \leq i \leq m\} \cup \{p_i = ([a_i:1], [0:1]) : m+1 \leq i \leq 2m\}.
\]
Then $\varphi_0$ lifts to an automorphism $\varphi$ of $X := \Bl_P (\CP^1 \times \CP^1)$ of order $2$. To see this, consider the rational map $\psi_0: \CP^1\times \CP^1 \dashrightarrow \CP^1$ given by projecting $\varphi_0$ to the second coordinate, i.e. 
\[
	\psi_0: ([X_1:X_2], [Y_1:Y_2]) \mapsto \left[Y_2 \prod_{i=m+1}^{2m} (X_1 - a_i X_2) : Y_1 \prod_{i=1}^m (X_1 - a_i X_2)\right].
\]
The set of basepoints of $\psi_0$ is equal to $P$ and the construction in the proof of \cite[Theorem II.7]{beauville} shows that $\psi_0$ extends to a rational morphism $X \to \CP^1$. Hence $\varphi_0$ extends to a birational morphism $\psi: X \to \CP^1\times \CP^1$. Finally, the universal property of blowups (\cite[Proposition II.8]{beauville}) shows that $\psi: X \to \CP^1 \times \CP^1$ extends to an automorphism $\varphi$ of $X$. See the following diagram; here, $b: X \to \CP^1\times \CP^1$ is the blowup of the points in $P$. 
\begin{figure}[H]
\centering
\begin{tikzcd}
\Bl_P(\CP^1\times\CP^1) \arrow[d, "b"] \arrow[rr, "\varphi"] \arrow[rrd, "\psi"] &  & \Bl_P(\CP^1\times\CP^1) \arrow[d, "b"] \\
\CP^1\times\CP^1 \arrow[rr, "\psi_0\times\Id = \varphi_0", dashed]               &  & \CP^1\times\CP^1                      
\end{tikzcd}
\end{figure}
One way to think about this extension $\varphi \in \Aut(X)$ is by restricting to the open and dense subset $V \subseteq X$ defined by
\[
	V := (\CP^1 - \{[a_i:1] :  1 \leq i \leq 2m\}) \times \CP^1.
\]
Because $\varphi_0$ restricts to an automorphism of $V$, the automorphism $\varphi$ is the unique continuous extension of $\varphi_0$ to $X$. 

Let $e_i$ denote the homology classes of the exceptional fibers over $p_i \in P$ for all $1 \leq i \leq 2m$ and let $S_1, S_2$ denote the homology classes of $X$ coming from the first and second factors of $\CP^1 \times \CP^1$ respectively. Then $H_2(X; \Z) = \Z\{S_1, S_2, e_1, \dots, e_{2m}\}$ with $Q_X(S_k, e_i) = 0$ and $Q_X(S_k, S_\ell) = 1-\delta_{k\ell}$ for all $k, \ell = 1, 2$ and $1 \leq i \leq 2m$.

Consider the projection map $\pr_0: \CP^1 \times \CP^1 \to \CP^1$ onto the first coordinate which extends to a map $\pr: X \to \CP^1$. Then $\pr\circ\varphi = \pr$ because $\pr_0\circ \varphi_0 = \pr_0$; see Figure \ref{fig:de-jonquieres} for an illustration. The fiber of $\pr$ over any $q \in \CP^1$ with $q \neq [a_i: 1]$ for all $i$ is $\pr^{-1}(q) = \{q\} \times \CP^1$ in $X$, which represents the homology class $S_2$. The map $\varphi$ restricts to a complex automorphism of each fiber $\pr^{-1}(q) = \{q\} \times \CP^1$ and so $\varphi_*(S_2) = S_2$. Over any $[a_i : 1] \in \CP^1$, the fiber $\pr^{-1}([a_i : 1])$ is a bouquet of two $\CP^1$, i.e. two copies of $\CP^1$ intersecting transversely at one point. One component is the exceptional fiber $e_i$ and the other component is the strict transform of the line $\pr_0^{-1}([a_i:1])$ in $X$. To determine the action of $\varphi$ on the exceptional fiber $e_i$, compute that $\varphi_0([a_i:1], [Y_1:Y_2]) = p_i \in P$ for any point of the form $([a_i:1], [Y_1:Y_2]) \in U$. Because $\varphi$ is an automorphism of $X$ of order $2$, this means that the strict transform of $\pr_0^{-1}([a_i:1])$ in $X$ must be sent to the exceptional fiber $e_i$ by $\varphi$ and vice versa. Hence $\varphi$ swaps the two components of $\pr^{-1}([a_i : 1])$. More explicitly, this means that $\varphi_*(e_i) = S_2 - e_i$, where $S_2 - e_i$ is the homology class of this strict transform of $\pr_0^{-1}([a_i:1])$.

\begin{figure}
\centering
\includegraphics{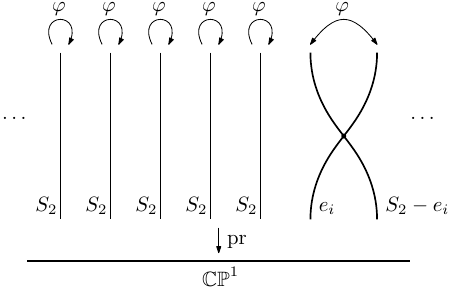}
\caption{The action of a de Jonqui\'eres involution $\varphi$ of $\Bl_P(\CP^1\times\CP^1)$ and the conic bundle $\pr: \Bl_P(\CP^1\times\CP^1) \to \CP^1$ preserved by $\varphi$. Each (vertical) line in the figure above represents a complex submanifold of $\Bl_P(\CP^1\times\CP^1)$ isomorphic to $\CP^1$; they are labelled with their respective homology classes. Each vertical connected component represents a fiber of the map $\pr$. The singular fibers are two copies of $\CP^1$ intersecting transversely at one point; there are $\lvert P \rvert$-many singular fibers.}\label{fig:de-jonquieres}
\end{figure}

The homological data described above determines the action of $\varphi_*$ on $H_2(M_n; \Z)$.
\begin{lem}\label{lem:de-jon-determined}
Let $n \geq 5$ be odd and let $g_1, g_2 \in \Mod(M_n)$. Consider some primitive $C_i \in H_2(M_n; \Z)$ and some $\Z$-submodule $N_i := \Z\{C_i, v_1^i, \dots, v_{n-1}^i\}$ of $H_2(M_n; \Z)$ for $i =1, 2$ such that the restriction of $Q_{M_n}$ to $N_i$ with respect to the given basis is $\langle 0 \rangle \oplus (n-1) \langle -1 \rangle$. If
\[
	g_i(C_i) = C_i, \qquad g_i(v_k^i) = C_i-v_k^i
\]
for all $1 \leq k \leq n-1$ and $i = 1, 2$ then $g_1$ and $g_2$ are conjugate in $\Mod(M_n)$. In particular, any such $g_1$ is conjugate to $[\varphi]$ where $\varphi$ is the de Jonqui\'eres involution on $X = \Bl_{P}(\CP^1\times\CP^1)$ defined above where $\lvert P\rvert = n-1$.
\end{lem}
\begin{proof}
Suppose the restriction of $Q_{M_n}$ to $\Z\{v_1^i, \dots, v_{n-1}^i\}^\perp \subseteq H_2(M_n; \Z)$ is odd for $i = 1$ or $2$. Then one can check that the restriction of $Q_{M_n}$ to $\Z\{C_i-v_1^i, v_2^i, \dots, v_{n-1}^i\}^\perp$ is even, and 
\[
	g_i(C_i-v_1^i) = g_i(C_i) - g_i(v_1^i) = C_i - (C_i - v_1^i).
\]
Furthermore, the restriction of $Q_{M_n}$ to $\Z\{C_i, C_i-v_1^i, v_2^i, \dots, v_{n-1}^i\}$ with respect to the given basis is $\langle 0 \rangle \oplus (n-1) \langle -1 \rangle$. Hence after possibly replacing $v_1^i$ with $C_i-v_1^i$, we may assume that the restriction of $Q_{M_n}$ to $\Z\{v_1^i, \dots, v_{n-1}^i\}^\perp$ is even for each $i$.

For each $i = 1, 2$, there is an orthogonal decomposition
\[
	H_2(M_n; \Z) = \Z\{v_2^i ,\,\dots,\, v_{n-1}^i\}\oplus \Z\{C_i,\, c_i',\, v_1^i\} 
\]
where $c_i' \in \Z\{v_2^i ,\,\dots,\, v_{n-1}^i\}^\perp$ is such that $Q_{M_n}(C_i, c_i') = 1$ which exists by unimodularity of $Q_{M_n}$ restricted to $\Z\{v_2^i ,\,\dots,\, v_{n-1}^i\}^\perp$. Denote $Q_{M_n}(c_i', v_1^i)$ by $A$ and let
\[
	c_i := \left(c_i' + A v_1^i\right) - \left(\frac{Q_{M_n}(c_i' + A v_1^i, c_i' + A v_1^i)}{2}\right)C_i.
\]
Note that $Q_{M_n}(c_i'+Av_1^i, c_i'+Av_1^i)$ is even because $c_i' + Av_1^i$ is in $\Z\{v_1^i, \dots, v_{n-1}^i\}^\perp$. Compute that that with respect to the $\Z$-basis $(C_i, c_i, v_1^i)$,
\[
	Q_{M_n}|_{\Z\{C_i,\, c_i,\, v_1^i\} } = \begin{pmatrix}
	0 & 1 & 0 \\
	1 & 0 & 0 \\
	0 & 0 & -1
	\end{pmatrix}.
\]
There is another orthogonal decomposition 
\[
	H_2(M_n; \Z) = \Z\{C_i-v_2^i,\, \dots,\, C_i-v_{n-1}^i\}\oplus \Z\{C_i,\, c_i-v_1^i-\dots-v_{n-1}^i,\, C_i-v_1^i\}.
\]
The only automorphism of $\Z\{C_i,\, c_i,\, v_1^i\}$ preserving $Q_{M_n}$ and fixing $C_i$ and $v_1^i$ is the identity. This uniquely determines $g_i$ since $g_i$ restricts to an isometry
\[
	g_i: \Z\{C_i,\,c_i,\,v_1^i\} \to \Z\{C_i,\, c_i-v_1^i-\dots-v_{n-1}^i,\, C_i-v_1^i\}
\]
with respect to the restrictions of $Q_{M_n}$ satisfying 
\[
	g_i(C_i) = C_i, \qquad g_i(v_1^i) = C_i-v_1^i.
\]
Finally let $\Phi \in \Mod(M_n)$ such that for all $1 \leq k \leq n-1$,
\[
	\Phi(v_k^1) = v_k^2, \qquad \Phi(C_1) = C_2, \qquad \Phi(c_1) = c_2.
\]
Then $g_1 = \Phi^{-1} \circ g_2\circ \Phi$.
\end{proof}

The birational involution $\varphi_0$ has (algebraic) degree $m+1$. Because $m = g+1 \geq 2$ in all constructions in this paper, any de Jonqui\'eres involution that we consider has degree $d > 2$. Moreover, $\varphi_0$ is birationally equivalent to the de Jonqui\'eres involutions of \cite[Example 2.4(c)]{bayle--beauville}. In the following lemma, we consider an explicit birational equivalence with an automorphism $f$ of a surface $\Bl_{P_0}\CP^2$. 
\begin{lem}\label{lem:de-jonquieres-cp2}
For any odd $n \geq 5$, there exist $P_0 \subseteq \RP^2 \subseteq \CP^2$ with $\lvert P_0 \rvert = n$ and an involution $f \in \Aut(\Bl_{P_0}\CP^2)$ conjugate to a de Jonquier\'es involution $\varphi_0$ described above in $\Cr(2)$ such that 
\begin{enumerate}
	\item $H_2(\Bl_{P_0}\CP^2; \Z) \cong H_2(\Bl_P(\CP^1\times\CP^1); \Z)$ as $\Z[G]$-modules with $G = \Z/2\Z$ acting by $\langle [f]\rangle$ and $\langle [\varphi]\rangle$ respectively,
	\item $f$ commutes with the anti-biholomorphism $\tau: \Bl_{P_0}\CP^2 \to \Bl_{P_0}\CP^2$ induced by complex conjugation on $\CP^2$, and
	\item $[f]$ is conjugate to $\prod_{k=1}^{\frac{n-1}{2}} \left(\RRef_{H-E_1-E_{2k}-E_{2k+1}} \circ \RRef_{E_{2k} - E_{2k+1}}\right)$ in $\Mod(M_n)$ after identifying $M_n \cong \Bl_{P_0}\CP^2$.
\end{enumerate}
\end{lem}
\begin{proof}
Let $a_1, \dots, a_{n-1} \in \R$ be distinct and let $\varphi_0$ and $P \subseteq \CP^1 \times \CP^1$ be defined as above. Fix some $p = ([a:1], [b:1])$  with $a, b \in \R_{\neq 0}$ and $a \neq a_i$ for all $i$ such that $\varphi_0(p) = p$ and let $q_1 = [0:0:1]$, $q_2 = [0:1:0]$. Consider the Hirzebruch surfaces $\mathbb F_0 = \CP^1\times \CP^1$ and $\mathbb F_1 = \Bl_{q_1}\CP^2$. There is an isomorphism $\Bl_p\mathbb F_0 \cong \Bl_{q_1, q_2}\CP^2$ which can be seen by explicitly writing
\[
	\Bl_{q_1, q_2}\CP^2 = \{([A:B], [C:D], [X:Y:Z]) : A(X+Y) - aBY = C(X+Z) - bDZ = 0\} \subseteq \CP^1\times\CP^1\times\CP^2
\]
and noting that the projection map onto the first two factors defines a blowup $\psi_1: \Bl_{q_1, q_2}\CP^2 \to \mathbb F_0$ which is an isomorphism onto $\mathbb F_0 - \{p\}$. Let $\psi_2: \Bl_{q_1, q_2} \CP^2 \to \mathbb F_1$ be a blowup given by projecting onto the first and third factors. The exceptional divisor of $\psi_2$ is the strict transform of $\pr_0^{-1}([a:1])$ in $\Bl_p\mathbb F_0$.

The rational map $F: \mathbb F_0 \dashrightarrow \mathbb F_1$ given by $F = \psi_2 \circ \psi_1^{-1}$ is the elementary transformation centered at $p$ (cf. \cite[Section 7.4.2]{dolgachev} or \cite[(2.5)]{bayle--beauville}) and is a morphism restricted to $\mathbb F_0 - \{p\}$. Note that $\psi_1^{-1}(P)$ is not contained in the exceptional divisor of $\psi_2$. Hence $F$ extends to $\Bl_P\mathbb F_0 \dashrightarrow \Bl_{F(P)} \mathbb F_1$; also denote this map by $F$. The maps $\psi_1$ and $\psi_2$ similarly extend and fit into the following commutative diagram:
\begin{figure}[H]
\centering
\begin{tikzcd}
& S := \Bl_{\{p\} \cup P} \mathbb F_0 \arrow[ld, "\psi_1"'] \arrow[rd, "\psi_2"] &                       \\
\Bl_P \mathbb F_0 \arrow[rr, "F", dashed] &                              & \Bl_{F(P)}\mathbb F_1
\end{tikzcd}
\end{figure}
Let $e \in H_2(S; \Z)$ denote the exceptional divisor over $p$. Let $\varphi$ be the automorphism of $\Bl_P \mathbb F_0$ induced by $\varphi_0$. Because $\varphi(p) = p$, the map $\varphi$ extends to an involution $\tilde \varphi$ of $S$. Because $\varphi$ preserves the fibers of $\pr$, the map $\tilde\varphi$ descends to an involution $f$ of $\Bl_{F(P)}\mathbb F_1$. Note that $f$ and $\varphi$ are conjugate in $\Cr(2)$ since $f = F \circ \varphi \circ F^{-1}$ as birational automorphisms of $\CP^2$. There are isometries
\begin{align*}
	(H_2(S; \Z), Q_S) &\cong (H_2(\Bl_P\mathbb F_0; \Z), Q_{\Bl_P\mathbb F_0}) \oplus (\Z\{e\}, Q_S|_{\Z\{e\}}) \\
	&\cong (H_2(\Bl_{F(P)}\mathbb F_1; \Z), Q_{\Bl_{F(P)}\mathbb F_1}) \oplus (\Z\{S_2-e\}, Q_S|_{\Z\{S_2-e\}})
\end{align*}
where the action of $[\tilde \varphi]$ on $H_2(S; \Z)$ restricts to the actions of $[\varphi]$ and $[f]$ on $H_2(\Bl_P\mathbb F_0; \Z)$ and $H_2(\Bl_{F(P)}\mathbb F_1; \Z)$ respectively. The $\Z[\langle [\tilde\varphi]\rangle]$-submodule $N := \Z\{S_2,\, e_1,\, \dots, e_{n-1}\}$ is contained in both $H_2(\Bl_P\mathbb F_0; \Z)$ and $H_2(\Bl_{F(P)}\mathbb F_1; \Z)$. Because the actions of $[f]$ and $[\varphi]$ agree on $N$, Lemma \ref{lem:de-jon-determined} shows there is an isometry
\[
	\iota: (H_2(\Bl_P\mathbb F_0; \Z), Q_{\Bl_P\mathbb F_0}) \to (H_2(\Bl_{F(P)}\mathbb F_1; \Z), Q_{\Bl_{F(P)}\mathbb F_1})
\]
which is also a $\Z[G]$-module isomorphism with $G = \Z/2\Z$ acting by $\langle [f] \rangle$ and $\langle [\varphi]\rangle$ respectively.

Note that $F \circ \tau_0 = \tau \circ F$ where $\tau_0: \mathbb F_0 \to \mathbb F_0$ and $\tau: \mathbb F_1 \to \mathbb F_1$ are diffeomorphisms induced by complex conjugation of the coordinates of $\CP^1$ and $\CP^2$ respectively. Then $\tau_0$ commutes with $\varphi$ and $F(P) \subseteq \RP^2 \subseteq \CP^2$ because $P \subseteq \mathbb F_0$ is pointwise fixed by $\tau_0$. Moreover, $\tau$ commutes with $F\circ \varphi \circ F^{-1}$, meaning that $f$ must commute with $\tau$ as a diffeomorphism of $\Bl_{F(P)}\CP^2 \to \Bl_{F(P)}\CP^2$. 

Consider
\[
	g = \prod_{k=1}^{\frac{n-1}{2}} \left(\RRef_{H- E_1 -E_{2k} - E_{2k+1}} \circ \RRef_{E_{2k} - E_{2k+1}}\right) \in \Mod(M_n).
\]
Note that each reflection in $g$ fixes $H-E_1$. Also,
\begin{align*}
	g(E_{2k+1}) &= \RRef_{H- E_1 -E_{2k} - E_{2k+1}} \circ \RRef_{E_{2k} - E_{2k+1}}(E_{2k+1}) = \RRef_{H- E_1 -E_{2k} - E_{2k+1}}(E_{2k}) = H-E_1-E_{2k+1},\\
	g(E_{2k}) &= \RRef_{H- E_1 -E_{2k} - E_{2k+1}} \circ \RRef_{E_{2k} - E_{2k+1}}(E_{2k}) = \RRef_{H- E_1 -E_{2k} - E_{2k+1}}(E_{2k+1}) = H-E_1-E_{2k}.
\end{align*}
Let $g_1 = g$ with $C_1 = H-E_1\in H_2(M_n; \Z)$ and $v_k^1 = E_{k+1} \in H_2(M_n; \Z)$ for all $1 \leq k \leq n-1$. Lemma \ref{lem:de-jon-determined} implies that $g$, $[f]$, and $[\varphi]$ are conjugate in $\Mod(M_n) \cong \Mod(\Bl_{F(P)}\mathbb F_1) \cong \Mod(\Bl_P \mathbb F_0)$.
\end{proof}
In the next two lemmas, we consider the action of $\varphi_*$ on $H_2(M_n; \Z)$. In this lemma and the rest of the paper, let $H_2(M_n; \Z)^G$ denote the subgroup fixed by $G$, i.e.
\[
	H_2(M_n; \Z)^G := \{c \in H_2(M_n; \Z) : g(c) = c \text{ for all }g \in G \}.
\]
\begin{lem}\label{lem:de-jonquieres-homology}
Let $n \geq 5$ be odd and let $\varphi \in \Aut(\Bl_{P}(\CP^1\times\CP^1))$ be the de Jonqui\'eres involution. Identify $\Bl_{P}(\CP^1\times\CP^1) \cong M_n$ and let $G = \langle [\varphi] \rangle \cong \Z/2\Z \leq \Mod^+(M_{n})$. As a $\Z[G]$-module, $H_2(M_{n};\Z) \cong \Z[G]^{\oplus 2}\oplus  C^{\oplus (n-3)}$ where $C \cong \Z$ as a $\Z$-module and $G$ acts by negation on $C$, and
\[
	H_2(M_n; \Z)^{G} = \Z\{S_2, \, 2S_1-e_1-\dots - e_{n-1} \}.
\]
\end{lem}
\begin{proof}
The fixed set of $\varphi$ is a smooth curve $\Gamma$ with a surjective morphism $\Gamma \to \CP^1$ of degree $2$ ramified over $(n-1)$-points (see \cite[Example 3.1]{blanc}); $\Gamma$ is a curve of genus $\frac{n-3}{2}$. There is an isomorphism 
\[
	H_2(M_n; \Z)\cong\Z^{\oplus t} \oplus C^{\oplus c}\oplus \Z[G]^{\oplus r}
\]
as $\Z[G]$-modules for some $t, r, c \in \Z$ by \cite[Proposition 1.1]{edmonds} where $C \cong \Z$ as $\Z$-modules and $\varphi_*$ acts by negation in $C$. By \cite[Proposition 2.4]{edmonds}, $\beta_0(\Gamma) + \beta_2(\Gamma) = t+2$ and $\beta_1(\Gamma) = c$ where $\beta_k(\Gamma)$ is the $k$th mod $2$ Betti number of $\Gamma$. Therefore $t = 0$ and $c = n-3$ so that 
\[
	H_2(M_n; \Z)\cong\Z[G]^{\oplus 2} \oplus C^{\oplus (n-3)}
\]
as $\Z[G]$-modules. As $\Q[G]$-modules,
\[
	H_2(M_n); \Q) \cong (C\otimes\Q)^{\oplus (n-1)} \oplus \Q^{\oplus 2}.
\]
A calculation shows that $S_2$ and $2S_1-e_1-\dots-e_{n-1}$ are fixed by $G$. Therefore,
\[
	H_2(M_n; \Z)^{G}= \Q\{S_2, \, 2S_1-e_1-\dots-e_{n-1} \} \cap H_2(M_n; \Z) = \Z\{S_2, \, 2S_1-e_1-\dots-e_{n-1} \}. \qedhere
\]
\end{proof}

\begin{lem}\label{lem:de-jonquieres-irreducible}
Let $n \geq 5$ be odd. If $\varphi$ and $f$ are de Jonqui\'eres involutions on some $\Bl_{P_0}(\CP^1\times\CP^1) \cong M_n$ and $\Bl_P\CP^2 \cong M_{n}$ respectively then $[\varphi], [f] \in \Mod^+(M_{n})$ are irreducible. 
\end{lem}
\begin{proof}
Suppose for some $1 \leq k \leq n$, there exists an isometric embedding 
\[
	\iota: (H_2(\#k\overline{\CP^2}; \Z), Q_{\#k\overline{\CP^2}}) \hookrightarrow (H_2(M_n; \Z), Q_{M_n})
\]
such that $\varphi_*$ restricts to an automorphism of the image. Let $v_1, \dots, v_k$ denote the orthogonal $\Z$-basis of $H_2(\#k\overline{\CP^2}; \Z)$; note that $Q_{M_n}(\iota(v_i), \iota(v_j)) = -\delta_{ij}$ for all $1 \leq i, j \leq k$. Because $\varphi_*$ acts as an element of $\OO(k)(\Z)$ on the image of $\iota$, 
\[
	\varphi_*(\iota(v_1)) = \iota(v_1), \, -\iota(v_1), \text{ or }\pm \iota(v_i) \text{ for some }i \neq 1.
\]
We address the three cases separately.
\begin{enumerate}
\item Suppose there exists some $c \in H_2(M_n; \Z)$ such that $Q_{M_n}(c,c) = -1$ and $\varphi_*(c) = c$. If $\varphi_*(c) = c$ then $c \in H_2(M_n; \Z)^G$. Compute that the restriction of $Q_{M_n}$ to $H_2(M_n; \Z)^G$ is
\[
	Q_{M_n}|_{H_2(M_n; \Z)^G} = \begin{pmatrix}
	0 & 2 \\
	2 & -(n-1)
	\end{pmatrix} = 2\begin{pmatrix}
	0 & 1 \\1 & -\frac{n-1}{2}
	\end{pmatrix}
\]
with respect to the $\Z$-basis of $H_2(M_n; \Z)^G$ given in Lemma \ref{lem:de-jonquieres-homology}. Therefore, $Q_{M_n}(x,x) \equiv 0 \pmod 2$ for all $x \in H_2(M_n; \Z)^G$. This is a contradiction because $Q_{M_n}(c, c) = -1$. 
\item Suppose there exists some $c \in H_2(M_n; \Z)$ such that $Q_{M_n}(c,c) = -1$ and $\varphi_*(c) = -c$. Then $c \in \Z\{S_2\}^\perp$ because $\varphi_*(S_2) = S_2$. The only elements of $x \in \Z\{S_2\}^\perp$ with $Q_{M_n}(x,x) = -1$ are of the form $x = aS_2 \pm e_k$ for some $a \in \Z$ and $1 \leq k \leq 2m$ because $\Z\{S_2\}^\perp= \Z\{S_2, e_1, \dots, e_{2m}\}$. On the other hand, if
\[
	-aS_2 \mp e_k = \varphi_*(aS_2 \pm e_k) = aS_2 \pm( S_2 - e_k)
\]
then $a \pm 1 = -a$. This is a contradiction since $a \in \Z$.
\item Suppose there exist some $c_1, c_2 \in H_2(M_n; \Z)$ such that $Q_{M_n}(c_k, c_\ell) = -\delta_{k\ell}$ and $\varphi_*(c_1) = c_2$. Then $Q_{M_n}(c_1-c_2, c_1-c_2) = -2$ and $\varphi_*(c_1-c_2) = -(c_1-c_2)$. Since $\varphi_*(S_2) = S_2$,
\[
	c_1-c_2 \in \Z\{S_2\}^\perp=\Z\{S_2, e_1, \dots, e_{n-1}\}.
\]
Then 
\[
	c_1-c_2 = aS_2 +(-1)^{a_k} e_k + (-1)^{a_j}e_j
\] 
some $a, a_k, a_j \in \Z$ and $1 \leq k, j \leq n-1$ because $Q_{M_n}(c_1-c_2, c_1-c_2) = -2$. Moreover, 
\[
	c_1 + c_2 \in H_2(M_n; \Z)^{G} = \Z\{S_2,\, 2S_1 - e_1 - \dots - e_{n-1}\}
\]
where the second equality holds by Lemma \ref{lem:de-jonquieres-homology}. However, for any $A, B \in \Z$,
\[
	(c_1-c_2) + (AS_2 + B(2S_1 - e_1 - \dots - e_{n-1})) \notin 2H_2(M_n; \Z).
\]
This is a contradiction because $(c_1-c_2) + (c_1+c_2) = 2c_1 \in 2H_2(M_n; \Z)$.
\end{enumerate}
Therefore, $[\varphi]$ is irreducible in $\Mod^+(M_n)$. Because $H_2(M_n; \Z)$ as a $\langle f_* \rangle$-module is isomorphic to $H_2(M_n; \Z)$ as a $\langle \varphi_*\rangle$-module by Lemma \ref{lem:de-jonquieres-cp2}, $[f] \in \Mod^+(M_n)$ is irreducible as well.
\end{proof}

\subsubsection{Geiser and Bertini involutions}\label{sec:geiser-bertini}
In this section we describe the Geiser involution $\gamma: X_7 \to X_7$ and the Bertini involution $\beta: X_8 \to X_8$ for any del Pezzo surface $X_n$ diffeomorphic to $M_n$ for $n = 7$ and $8$; we follow the exposition of \cite{bayle--beauville}.

Let $X_7 = \Bl_P\CP^2$ with $P$ a set of $7$ points in general position in $\CP^2$. For any $p \in \CP^2 - P$, the pencil of cubic curves passing through the points $P \cup \{p\}$ has a ninth base point $q$. The map $\gamma: p \mapsto q$ defines a birational map $\gamma: \CP^2 \dashrightarrow \CP^2$ and induces an order $2$ automorphism of $X_7$, which we also denote by $\gamma$. Another way to construct this map is to consider the linear system $\lvert -K_{X_7}\rvert$ which defines a double covering $f: X_7 \to \CP^2$ branched along a smooth curve $C$ of genus $3$. Then $\gamma$ is the nontrivial deck transformation of this branched cover, and the fixed set $\Fix(\gamma)$ in $X_7$ is $C$.

Let $X_8 = \Bl_P\CP^2$ with $P$ a set of $8$ points in general position in $\CP^2$. Consider the linear system $\lvert -2K_{X_8} \rvert$ which defines a double covering $f: X_8 \to Q$ onto a quadric cone $Q \subseteq \CP^3$ branched along the vertex $v$ of $Q$ and a smooth curve $C$ of genus $4$. Then $\beta$ is the nontrivial deck transformation of this branched cover, and the fixed set $\Fix(\gamma)$ in $X_8$ is $C \sqcup \{q\}$, where $q$ is the ninth base point of the pencil of cubics defined by $P$.

By \cite[p. 410]{dolgachev}, the Geiser involution $\gamma$ acts on the subgroup $\mathbb E_7 = \Z\{K_{X_7}\}^\perp$ of $H_2(X_7; \Z) \cong H^2(X_7; \Z) \cong \Pic(X_7)$ by negation and is the product of seven, pairwise-commuting involutions in $W_7$. By \cite[p. 414]{dolgachev}, the Bertini involution $\beta$ acts on the subgroup $\mathbb E_8 = \Z\{K_{X_8}\}^\perp$ of $H_2(X_8;\Z) \cong H^2(X_8; \Z) \cong \Pic(X_8)$ by negation and is the product of eight, pairwise-commuting involutions in $W_8$. In particular, $\gamma$ and $\beta$ fix $K_{X_7}$ and $K_{X_8}$ respectively. 

We conclude this section by noting that $[\gamma]$ and $[\beta]$ are irreducible elements of $\Mod^+(M_n)$ for $n = 7$, $8$ respectively.
\begin{lem}\label{lem:geiser-bertini-irreducible}
The mapping classes $[\gamma] \in \Mod^+(M_7)$ and $[\beta]\in \Mod^+(M_8)$ are irreducible.
\end{lem}
\begin{proof}
Let $G_7 = \langle [\gamma] \rangle$ and $G_8 = \langle [\beta]\rangle$. 
\begin{enumerate}
	\item If $v \in H_2(M_7; \Z)$ is fixed by $G_7$ then consider $H_2(M_7; \Q)^{G_7}$, the subspace of $H_2(M_7; \Q)$ that is pointwise fixed by $G_7$. There is a decomposition of $H_2(M_7; \Q)$ as a $\Q[G_7]$-module
	\[
		H_2(M_7; \Q) = \Q\{K_{X_7}\} \oplus \mathbb E_7 \otimes \Q
	\]
	and $H_2(M_7; \Q)^{G_7} = \Q\{K_{X_7}\}$. Taking intersections with $H_2(M_7; \Z)$ on both sides shows that $H_2(M_7; \Z)^{G_7} = \Z\{K_{X_7}\}$. 
	\item The restriction of $Q_{M_8}$ to $\Z\{K_{X_8}\}$ is unimodular so there is an orthogonal decomposition of $H_2(M_8; \Z)$ as a $\Z[G_8]$-module as 
	\[
		H_2(M_8; \Z) \cong \mathbb E_8 \oplus \Z\{K_{X_8}\} \cong C^{\oplus 8} \oplus \Z.
	\] 
\end{enumerate}
In both cases, $H_2(M_n; \Z)^{G_n} = \Z\{K_{X_n}\}$. If $g = [\gamma]$ or $g = [\beta]$ is reducible, there are two possibilities:
\begin{enumerate}
	\item There exists some $c \in H_2(M_n; \Z)$ such that $Q_{M_n}(c,c) = -1$ and $g(c) = \pm c$. If $g(c) = c$ then $c = aK_{X_n}$ for some $a \in \Z$. However,
	\[
		Q_{M_n}(aK_{X_n}, aK_{X_n}) = a^2(9-n) \geq 0.
	\]
	Because $Q_{M_n}(c,c) = -1$, this is a contradiction. If $g(c) = -c$ then $c \in \mathbb E_n$. This is a contradiction because $\mathbb E_n$ is an even lattice.
	\item There exist some $c_1, c_2 \in H_2(M_n; \Z)$ such that $Q_{M_n}(c_k, c_\ell) = -\delta_{k\ell}$ and $g(c_1) = c_2$. Then $c_1+c_2 \in H_2(M_n; \Z)^{G_n}$, meaning that $c_1+c_2 = aK_{X_n}$ for some $a \in \Z$ and
	\[
		Q_{M_n}(aK_{X_n}, aK_{X_n}) = a^2(9-n) \geq 0.
	\]
	Because $Q_{M_n}(c_1+c_2, c_1+c_2) = -2$, this is a contradiction. \qedhere
\end{enumerate}
\end{proof}

\subsection{Involutions in $\Mod^+(M_n)$ for $5 \leq n \leq 8$}\label{sec:n58}

With the discussion of conjugacy classes of involutions in the plane Cremona group above, we are ready to continue analyzing the cases $5 \leq n \leq 8$. The next lemma gives a criterion for reducibility of mapping classes $g \in \Mod(M_n)$. 

\begin{lem}\label{lem:lattices}
Let $g \in \Mod(M_n)$. Then $g$ is reducible in the following cases:
\begin{enumerate}
\item if $n \leq 7$ and there exists $c \in H_2(M_n; \Z)$ with $Q_{M_n}(c,c) = 1$ such that $g(c) = c$; \label{lem:lattices1}
\item if $n \leq 8$ and there exists $c_1, c_2 \in H_2(M_n; \Z)$ with $Q_{M_n}(c_k, c_\ell) = 1-\delta_{k\ell}$ such that $g(c_1) = c_2$ or $g(c_k) = c_k$ for $k = 1, 2$;\label{lem:lattices3}
\item if $n \leq 9$ and there exists $c \in H_2(M_n; \Z)$ with $Q_{M_n}(c,c) = -1$ such that $g(c) = c$. \label{lem:lattices2}
\end{enumerate}
\end{lem}
\begin{proof}
\begin{enumerate}
	\item The restriction of the intersection form $Q_{M_n}$ to $\Z\{c\}^\perp$ is unimodular and negative-definite by \cite[Lemma 1.2.12]{gompf--stipsicz}. For $n \leq 7$, there is only one unimodular and negative-definite symmetric form of rank $n$, and so there is an isometry 
	\[
		\iota: H_2(\CP^2; \Z) \oplus H_2(\#n \overline{\CP^2}; \Z) \to H_2(M_n; \Z)
	\]
	such that $\iota(H) = c$ and the image of $H_2(\#n \overline{\CP^2}; \Z)$ under $\iota$ is $\Z\{c\}^\perp$. Then $g$ is contained in the image of $\iota_*$ because $g$ preserves $\Z\{c\}^\perp$ and $\Z\{c\}$. Therefore, $g$ is reducible.

	\item The restriction of the intersection form $Q_{M_n}$ to $\Z\{c_1, c_2\}^\perp$ is unimodular and negative-definite with rank $n-1$ by \cite[Lemma 1.2.12]{gompf--stipsicz}. Because $n-1 \leq 7$, there is only one unimodular and negative-definite symmetric form of rank $n-1$. Therefore, $g$ is reducible by the same reasoning as the proof of (\ref{lem:lattices1}). 

	\item The restriction of the intersection for $Q_{M_n}$ to $\Z\{c\}^\perp$ is unimodular and indefinite with signature $(1,n-1)$ by \cite[Lemma 1.2.12]{gompf--stipsicz}. If $n \neq 2$ then the signature $\sigma(Q_{M_n}|_{\Z\{c\}^\perp}) = 2-n$ is not divisible by $8$, so the lattice $(\Z\{c\}^\perp, Q_{M_n}|_{\Z\{c\}^\perp})$ is odd by \cite[Lemma 1.2.20]{gompf--stipsicz}. There is an isometry 
	\[
		\iota: H_2(M_{n-1}; \Z) \oplus H_2(\overline{\CP^2}; \Z) \to H_2(M_n; \Z)
	\]
	such that the image of $H_2(M_{n-1};\Z)$ under $\iota$ is $\Z\{c\}^\perp$ and the image of $H_2(\overline{\CP^2}; \Z)$ under $\iota$ is $\Z\{c\}$ by \cite[Theorem 1.2.21]{gompf--stipsicz}. If $n = 2$ then the signature $\sigma(Q_{M_n}|_{\Z\{c\}^\perp}) = 0$. So $(\Z\{c\}^\perp, Q_{M_n}|_{\Z\{c\}^\perp}) \cong (H_2(M; \Z), Q_M)$ for $M = M_*$ or $M_1$ because these are the only two indefinite lattices of rank $2$. There is an isometry
	\[
		\iota: H_2(M; \Z) \oplus H_2(\overline{\CP^2}; \Z) \to H_2(M_2; \Z)
	\]
	such that the image of $H_2(M; \Z)$ under $\iota$ is $\Z\{c\}^\perp$ and the image of $H_2(\overline{\CP^2}; \Z)$ under $\iota$ is $\Z\{c\}$ by \cite[Theorem 1.2.21]{gompf--stipsicz}. 

	Therefore, $g$ is contained in the image of $\iota_*$, so $g$ is reducible. \qedhere
\end{enumerate}
\end{proof}

One way to determine all conjugacy classes of $W_5 = W(D_5)$ of order $2$ is to consult \cite[Table 3]{carter}, but we apply \cite[Lemma 5]{carter} instead. To do so, we consider the \emph{roots} of $\mathbb E_n$. For each $5 \leq n \leq 8$, let
\[
	\mathcal R_n := \{e \in \mathbb E_n : Q_{M_n}(e,e) = -2\}. 
\]
Then $\mathcal R_n$ is finite by \cite[Proposition 8.2.7]{dolgachev}. One can check that $\mathcal R_n$ is a root system for the Euclidean space $\mathbb E_n \otimes \R$ with the bilinear form $-Q_{M_n}$ extended $\R$-linearly. Any element of $\mathcal R_n$ is called a \emph{root} of $\mathbb E_n$. For $n = 5$, we first determine the maximal set of mutually orthogonal roots of $\mathbb E_5$, up to $W_5$-action. 
\begin{lem}
Up to $W_5$-action and up to sign, the unique maximal set of mutually orthogonal roots of $\mathbb E_5$ is
\begin{equation}\label{eqn:maximal-ortho}
	S = \{H-E_1-E_2-E_3,\, H-E_1-E_4-E_5, \, E_2-E_3, \, E_4-E_5\}.
\end{equation}
\end{lem}
\begin{proof}
Let $S$ be a maximal set of mutually orthogonal roots of $\mathbb E_5$. By \cite[Proposition 8.2.7]{dolgachev}, the roots of $\mathbb E_5$ are of the form $E_i-E_j$ and $\pm(H-E_i-E_j-E_k)$ for $i, j, k$ distinct. The group $W_5$ acts transitively on the roots by \cite[Proposition 8.2.17]{dolgachev}, so we may assume that $\alpha_1 = H-E_1-E_2-E_3 \in S$. 
\begin{enumerate}
	\item Suppose $\alpha_2 = H-E_i-E_j-E_k \in S$ with $\alpha_1 \neq \alpha_2$. Because $Q_{M_5}(\alpha_1, \alpha_2) = 0$, up to relabeling the vectors $E_i$, we may assume that $\alpha_2 = H-E_1-E_4-E_5$. No other roots of the form $H-E_i-E_j-E_k$ are orthogonal to both $\alpha_1$ and $\alpha_2$.

	If $\alpha_3 = E_i-E_j \in S$, then $\{i,j\} = \{2,3\}$ or $\{4,5\}$. Since $E_2-E_3$ and $E_4-E_5$ are orthogonal, we see that the set $S$ as given in (\ref{eqn:maximal-ortho}) is the unique maximal set containing multiple roots of the form $H - E_1-E_2-E_3$.

	\item Suppose there are no other roots of the form $H-E_i-E_j-E_k \in S$. If $E_i - E_j \in S$, either $\{i,j\} = \{4,5\}$ of $\{i,j\} \subseteq \{1,2,3\}$. Without loss of generality, we may assume that $E_2-E_3, E_4-E_5 \in S$. No other roots of the form $E_i-E_j$ are orthogonal to all elements of $S$. This set $S$ is then contained in the maximal set given in (\ref{eqn:maximal-ortho}).  \qedhere
\end{enumerate}
\end{proof}
\begin{prop}\label{prop:n5}
There is exactly one conjugacy class of irreducible involutions in $\Mod^+(M_5)$, and the elements of this conjugacy class are realized by de Jonqui\'eres involutions of (algebraic) degree $3$.
\end{prop}
\begin{proof}
The group $W_5$ is the Weyl group $W(D_5) \cong (\Z/2\Z)^4 \rtimes S_5$. Consider various subsets $I$ of the maximal mutually orthogonal set of $S$ as in (\ref{eqn:maximal-ortho}), up to $W_5$-orbits. For each $I \subseteq S$, consider $g = \prod_{\alpha \in I} \RRef_\alpha \in W_5$. 
\begin{enumerate}
	\item If $I = S$ then
	\[
		g = (\RRef_{H-E_1-E_2-E_3}\circ \RRef_{E_2-E_3}) \circ (\RRef_{H-E_1-E_4-E_5} \circ \RRef_{E_4-E_5}).
	\]
	By Lemmas \ref{lem:de-jonquieres-cp2} and \ref{lem:de-jonquieres-irreducible}, $g$ is irreducible and realized by de Jonqui\'eres involutions of degree $3$.
	\item If $I = \{H-E_1-E_2-E_3, \, H-E_1-E_4-E_5\}$ then $g(c)=c$ with $c = 2H-E_1-E_3-E_5$ because $c \in \Z\{I\}^\perp$. Because $Q_{M_5}(c,c) = 1$, Lemma \ref{lem:lattices}(\ref{lem:lattices1}) implies that $g$ is reducible.

	\item If $I \subseteq \{H-E_1-E_2-E_3, \, E_2-E_3, \, E_4-E_5\}$, then $g$ is in the image of the standard inclusion
	\[
		\iota_*: \Mod(M_3)\times \Mod(\#2\overline{\CP^2}) \hookrightarrow \Mod(M_5).
	\]
	Therefore, $g$ is reducible.

	\item If $I = \{H-E_1-E_2-E_3, \, H-E_1-E_4-E_5, \, E_2-E_3\}$ then consider
	\[
		\alpha_1 := H-E_5 \in \Z\{H-E_1-E_4-E_5, \, E_2-E_3\}^\perp.
	\]
	Let $\alpha_2 := 2H-E_1-E_2-E_3-E_5$ and compute that 
	\[
		g(\alpha_1) = \RRef_{H-E_1-E_2-E_3}(\alpha_1) = \alpha_2.
	\]
	By Lemma \ref{lem:lattices}(\ref{lem:lattices3}) , $g$ is reducible because $g(\alpha_1) = \alpha_2$.
\end{enumerate}
Any element of order $2$ in $W_5$ can be written as a product of reflections about mutually orthogonal roots by \cite[Lemma 5]{carter}. Hence, we have shown that there is a unique irreducible conjugacy class of order $2$ in $W_5$ and this class is realized by a de Jonqui\'eres involution of degree $3$. Corollary \ref{cor:involution-classification} then implies that this is the only irreducible conjugacy class of order $2$ in $\Mod^+(M_5)$.
\end{proof}

For the rest of this paper, we use the list of conjugacy classes of each $W_n$ given in \cite{carter}. The classification of \cite{carter} is stated in terms of a graph $\Gamma$ (called \emph{Carter graph}) that one can associate to each conjugacy class $C$ of $W_n$ (cf. \cite[p. 6]{carter}). We briefly describe the Carter graph of a conjugacy class $C$ of order $2$ here:

First, let $g \in W_n$ be any element and let $C$ be its conjugacy class. The graph $\Gamma$ of $C$ depends a priori on a factorization $g = w_1w_2$ where $w_1, w_2\in W_n$ each have order $2$. To construct $\Gamma$, consider a set of roots $\{v_i : 1 \leq i \leq k+h\} \subseteq \mathbb E_n$ such that
\[
	w_1 = \prod_{i=1}^{k}\RRef_{v_i}, \qquad w_2 = \prod_{i=k+1}^{k+h}\RRef_{v_i}
\]
and 
\begin{enumerate}
\item $V_1 \cap V_2 = 0$, where $V_i \subseteq \mathbb E_n \otimes \R$ denotes the $(-1)$-eigenspace of $w_i$ for $i = 1, 2$,
\item the roots in $R_1 := \{v_1, \dots, v_k\}$, which span $V_1$, are mutually orthogonal (with respect to $\R$-bilinear extension of $Q_{M_n}$ restricted to $\mathbb E_n$), and
\item the roots in $R_2 := \{v_{k+1}, \dots, v_{k+h}\}$, which span $V_2$, are mutually orthogonal.
\end{enumerate}
The existence of such a factorization of $g$ is guaranteed by \cite[Proposition 38, Corollary (ii)]{carter}. Finally, let $\Gamma$ be the graph with vertex set $\{v_1, \dots, v_{k+h}\}$ and with $\left(\frac{4 Q_{M_n}(v_i, v_j)^2}{Q_{M_n}(v_i, v_i) Q_{M_n}(v_j, v_j)}\right)$-many edges between the vertices $v_i$ and $v_j$. 

Now assume that $g$ has order $2$ and consider any factorization $g = w_1w_2$ as above. Because $w_1$ and $w_2$ each have order $2$ and commute, they are simultaneously diagonalizable in $\GL(\mathbb E_n \otimes \R)$. Hence there is an orthogonal decomposition $V_2 \cong (V_2 \cap V_1) \oplus (V_2\cap V_1^\perp)$ since $V_1^\perp$ is precisely the $(1)$-eigenspace of $w_1$ in $\mathbb E_n \otimes \R$. Because $V_2 \cap V_1 = 0$, there is an inclusion $V_2 \subseteq V_1^\perp$. Finally this shows that the roots in $R_1 \cup R_2$ are mutually orthogonal. Moreover, the number of roots in $R_1 \cup R_2$ is equal to the dimension of the $(-1)$-eigenspace of $g$ acting on $\mathbb E_n \otimes \R$. 

Therefore in the case of conjugacy classes $C$ of order $2$, Carter's construction is independent of the choices of $w_i$ and the factorization of each $w_i$ into reflections about mutually orthogonal roots. So if $g \in C$ is a product of reflections $\RRef_{v_k}$ about mutually orthogonal roots $v_k$ with $1 \leq k\leq m$ then the unique Carter graph of $C$ is $\Gamma = (A_1)^m$ which has $m$ vertices and no edges. This is the Dynkin diagram of the Weyl subgroup $\langle \RRef_{v_1}, \dots, \RRef_{v_m} \rangle \cong W(A_1)^m$ of $W_n$. Moreover, $m$ is the dimension of the $(-1)$-eigenspace of $g$ acting on $\mathbb E_n \otimes \R$. 

Throughout the rest of this section, we use the notation
\[
	\alpha_{ijk} := H - E_i - E_j - E_k \in H_2(M_n; \Z)
\]
for $1 \leq i, j, k \leq n$ distinct.
\begin{lem}\label{lem:n6}
Any element $g \in \Mod^+(M_6)$ of order $2$ is reducible.
\end{lem}
\begin{proof}
Consider the set
\[
	S = \{\alpha_{123},\, \alpha_{145},\, E_2-E_3, \, E_4-E_5\}
\]
of four mutually orthogonal roots of $\mathbb E_6$. According to \cite[Table 9]{carter}, the conjugacy classes of order $2$ are in bijection with the graphs $\Gamma = (A_1)^m$ with $1 \leq m \leq 4$. Therefore, such conjugacy classes are represented by elements of the form $\prod_{\alpha \in I}\RRef_{\alpha}$ for some $I \subseteq S$. All such involutions $g$ satisfy $g(E_6) = E_6$, making them reducible by Lemma \ref{lem:lattices}(\ref{lem:lattices2}) . 
\end{proof}

\begin{prop}\label{prop:n7}
There are two conjugacy classes of irreducible involutions in $\Mod^+(M_7)$ and the elements of these conjugacy classes are realized by de Jonqui\'eres involutions of (algebraic) degree $4$ and Geiser involutions.
\end{prop}
\begin{proof}
Consider the set of mutually orthogonal roots
\[
	S = \{\alpha_{127}, \, \alpha_{347}, \, \alpha_{567}, \, E_1-E_2, \, E_3-E_4, \, E_5-E_6, \, 2H - E_1 - \dots- E_6\}.
\]
According to \cite[Table 10]{carter}, the Carter graphs of the conjugacy classes of order $2$ are of the form $\Gamma = (A_1)^k$ for some $1 \leq k \leq 7$. Each graph 
\[
	\Gamma = A_1, \, (A_1)^2, \, (A_1)^5, \, (A_1)^6, \, (A_1)^7
\]
has a unique associated conjugacy class of order $2$ in $W_7$. Each graph
\[
	\Gamma = (A_1)^3, \, (A_1)^4
\]
has two associated conjugacy classes of order $2$ in $W_7$. 
\begin{enumerate}
	\item The conjugacy classes of $\Gamma = A_1$ and $\Gamma = (A_1)^2$ are represented by $g=\RRef_{E_1-E_2}$ and $g = \RRef_{E_1-E_2}\circ \RRef_{E_3-E_4}$ respectively. In both cases, $g$ is reducible by Lemma \ref{lem:lattices}(\ref{lem:lattices1}) because $g(H) = H$.
	\item The conjugacy class of $\Gamma = (A_1)^5$ is represented by $g = \prod_{\alpha \in I} \RRef_\alpha$ with
	\[
		I = \{\alpha_{123},\, \alpha_{145},\, E_2-E_3, \, E_4-E_5, \, E_6-E_7 \}. 
	\]
	Then $g$ is in the image of the standard inclusion
	\[
		\iota_*: \Mod^+(M_5) \times \Mod(\#2\overline{\CP^2}) \hookrightarrow \Mod^+(M_7).
	\]
	Therefore, $g$ is reducible.

	\item The conjugacy class of $\Gamma = (A_1)^6$ is represented by $g = \prod_{\alpha \in I} \RRef_\alpha$ with 
	\[
		I = \{\alpha_{127}, \, \alpha_{347}, \, \alpha_{567},\, E_1-E_2, \, E_3-E_4, \, E_5-E_6\}.
	\]
	By Lemmas \ref{lem:de-jonquieres-cp2} and \ref{lem:de-jonquieres-irreducible}, $g$ is irreducible and realized by a de Jonqui\'eres involution of degree $4$.

	\item The conjugacy class of $\Gamma = (A_1)^7$ is represented by $g = \prod_{\alpha \in S} \RRef_\alpha$. Then $g$ acts by negation on $\mathbb E_7$ and is realized by the Geiser involution as described in Section \ref{sec:geiser-bertini}. By Lemma \ref{lem:geiser-bertini-irreducible}, $g$ is irreducible. 
\end{enumerate}
The remaining two cases are $\Gamma =(A_1)^3$ and $(A_1)^4$. 
\begin{enumerate}
	\item There are two conjugacy classes of order $2$ associated to $\Gamma = (A_1)^3$. Consider the two elements
	\begin{align*}
		h_1 &= \RRef_{\alpha_{127}}\circ \RRef_{\alpha_{347}} \circ \RRef_{\alpha_{567}}, \\
		h_2 &= \RRef_{E_1-E_2} \circ \RRef_{E_3-E_4} \circ \RRef_{E_5-E_6}.
	\end{align*}
	Let $\alpha_1 = 2H-E_1-E_3-E_5-E_7$ and $\alpha_2 = H - E_7$ and note that $h_1(\alpha_i) = \alpha_i$ for each $i = 1, 2$ because
	\[
		\alpha_1,\, \alpha_2 \in \Z\{\alpha_{127},\, \alpha_{347},\, \alpha_{567}\}^\perp.
	\]
	Because $Q_{M_7}(\alpha_k, \alpha_\ell) = 1-\delta_{k\ell}$, Lemma \ref{lem:lattices}(\ref{lem:lattices3}) shows that $h_1$ is reducible. Moreover, $h_2$ is reducible by Lemma \ref{lem:lattices}(\ref{lem:lattices1}) because $h_2(H) = H$; the subspace fixed by $h_2$ is
	\[
		H_2(M_7; \Z)^{\langle h_2 \rangle} = \Z\{H,\, E_7\} \oplus \Z\{E_1+E_2,\, E_3+E_4, \, E_5+E_6 \}.
	\]

	Suppose $h_1$ and $h_2$ are conjugate in $\Mod(M_7)$, so that there exist some $c_1, c_2 \in H_2(M_7; \Z)^{\langle h_2 \rangle}$ such that $Q_{M_7}(c_i, c_j) = Q_{M_7}(\alpha_i, \alpha_j)$ for all $i, j$. Then
	\[
		c_i = A_i E_7 + \left(\sum_{k=1}^3 B_{i, k} (E_{2k-1}+E_{2k})\right) + C_i H
	\]
	for some $A_i, B_{i,k}, C_i \in \Z$ for $i = 1, 2$ and $k = 1, 2, 3$ with $C_i^2 = A_i^2 + 2\sum_{k=1}^3 B_{i,k}^2$. Taking both sides mod $2$, we see that $C_i \equiv A_i \pmod 2$ for $i= 1, 2$ so that $C_1C_2-A_1A_2 \equiv 0\pmod 2$. However, $Q_{M_7}(c_1, c_2) = 1$ and
	\[
		Q_{M_7}(c_1, c_2) = -A_1A_2 + \left(\sum_{k=1}^3 -2B_{1,k}B_{2,k}\right) + C_1C_2 \equiv -A_1A_2+C_1C_2 \pmod 2.
	\]
	This is a contradiction. Therefore, both $h_1$ and $h_2$ are reducible and are not conjugate to each other in $\Mod(M_7)$.

	\item There are two conjugacy classes of order $2$ associated to $\Gamma = (A_1)^4$. Consider the two elements
	\begin{align*}
	h_1 &= \RRef_{E_1-E_2} \circ \RRef_{E_3-E_4} \circ \RRef_{E_6-E_7} \circ \RRef_{H-E_1-E_2-E_5} \\
	h_2 &= (\RRef_{H-E_1-E_2-E_3} \circ \RRef_{E_2-E_3}) \circ (\RRef_{H-E_1-E_4-E_5} \circ \RRef_{E_4-E_5}). 
	\end{align*}
	Then $h_1$ and $h_2$ are in the image of the standard inclusion
	\[
		\iota_*: \Mod^+(M_5) \times \Mod(\#2\overline{\CP^2}) \hookrightarrow \Mod^+(M_7)
	\]
	because $h_1$ and $h_2$ both preserve $\Z\{E_6, E_7\}$. Therefore, $h_1$ and $h_2$ are both reducible. 

	By Lemmas \ref{lem:de-jonquieres-cp2} and \ref{lem:de-jonquieres-irreducible}, the restriction of $h_2$ to $\iota(H_2(M_5; \Z))$ is irreducible and realizable by a de Jonqui\'eres involution. Moreover, $h_2$ restricts to a trivial action on $\iota(H_2(\#2\overline{\CP^2}; \Z))$. By Lemma \ref{lem:de-jonquieres-homology}, there is a decomposition as a $\Z[\langle h_2\rangle]$-module
	\[
		H_2(M_7; \Z) \cong \iota(H_2(M_5; \Z)) \circ \iota(H_2(\#2\overline{\CP^2}; \Z)) \cong \Z[\langle h_2\rangle]^{\oplus 2} \oplus C^{\oplus 2} \oplus \Z^{\oplus 2}
	\]
	where $C \cong \Z$ as a $\Z$-module and $h_2$ acts by negation on $C$. On the other hand, the $\Z[\langle h_1\rangle]$-module structure of $H_2(M_7; \Z)$ is
	\[
		H_2(M_7; \Z) = \Z\{ H-E_1,\, H-E_2\} \oplus \Z\{ H-E_1-E_2,\, E_5 \} \oplus \Z\{E_3,\, E_4\} \oplus \Z\{E_6,\, E_7\} \cong \Z[\langle h_1 \rangle]^{\oplus 4}.
	\]
	If $h_1$ and $h_2$ are conjugate in $\Mod(M_7)$ then the $\Z[\langle h_1 \rangle]$- and $\Z[\langle h_2 \rangle]$-module structures of $H_2(M_7; \Z)$ agree. Therefore, $h_1$ and $h_2$ are not conjugate in $W_7$ and the two conjugacy classes of order $2$ associated to $\Gamma = (A_1)^4$ are represented by $h_1$ and $h_2$.
\end{enumerate}
Therefore, the conjugacy classes of the Carter graphs $\Gamma = (A_1)^6$ and $(A_1)^7$ are the only two irreducible conjugacy classes of order $2$ in $W_7$ and they are realized by a de Jonqui\'eres involution $f$ of degree $4$ and a Geiser involution $\gamma$ respectively. By Lemmas \ref{lem:de-jonquieres-homology} and \ref{lem:geiser-bertini-irreducible},
\[
	H_2(M_7; \Z)^{\langle f \rangle} \cong \Z^2, \qquad H_2(M_7; \Z)^{\langle \gamma \rangle} \cong \Z
\] 
so $[f]$ and $[\gamma]$ are not conjugate in $\Mod(M_7)$. Corollary \ref{cor:involution-classification} then implies that these are the only two irreducible conjugacy classes of order $2$ in $\Mod^+(M_7)$.
\end{proof}

\begin{prop}\label{prop:n8}
There is exactly one conjugacy class of irreducible involutions in $\Mod^+(M_8)$ and the elements of this conjugacy class are realized by Bertini involutions.
\end{prop}
\begin{proof}
According to \cite[Table 11]{carter}, the Carter graphs of the conjugacy classes of $W_8$ of order $2$ are of the form $\Gamma = \Gamma(A_1)^k$ for some $1 \leq k \leq 8$. Each graph 
\[
	\Gamma = A_1, \, (A_1)^2, \, (A_1)^3, \, (A_1)^5, \, (A_1)^6, \, (A_1)^7, \, (A_1)^8
\]
has a unique associated conjugacy class of order $2$ in $W_8$. The graph
\[
	\Gamma = (A_1)^4
\]
has two associated conjugacy classes of order $2$ in $W_8$. 
\begin{enumerate}
	\item The conjugacy classes of $\Gamma = (A_1)^k$ for $1 \leq k \leq 7$ and $k \neq 4$ are represented by $g = \prod_{\alpha \in I} \RRef_\alpha$ for some 
	\[
		I \subseteq \{\alpha_{127}, \, \alpha_{347}, \, \alpha_{567}, \, E_1-E_2, \, E_3-E_4, \, E_5-E_6, \, 2H - E_1 - \dots - E_6\}.
	\]
	Then $g(E_8) = E_8$ and therefore $g$ is reducible by Lemma \ref{lem:lattices}(\ref{lem:lattices2}).

	\item There are two conjugacy classes of order $2$ associated to $\Gamma = (A_1)^4$. Consider
	\begin{align*}
	h_1 &= \RRef_{E_1-E_2} \circ \RRef_{E_3-E_4} \circ \RRef_{E_6-E_7} \circ \RRef_{H-E_1-E_2-E_5} \\
	h_2 &= (\RRef_{H-E_1-E_2-E_3} \circ \RRef_{E_2-E_3}) \circ (\RRef_{H-E_1-E_4-E_5} \circ \RRef_{E_4-E_5}). 
	\end{align*}
	Both $h_1$ and $h_2$ are reducible by Lemma \ref{lem:lattices}(\ref{lem:lattices2}) because $h_i(E_8) = E_8$ for $i = 1, 2$. By the same proof as in the analogous case in Proposition \ref{prop:n7}, $h_1$ and $h_2$ are not conjugate in $\Mod(M_8)$. Therefore, the two conjugacy classes of order $2$ associated to $\Gamma = (A_1)^4$ are represented by $h_1$ and $h_2$.

	\item The conjugacy class of $\Gamma = (A_1)^8$ is represented by $g$ which acts by negation on $\mathbb E_8$ and fixes $3H - E_1 - \dots - E_8$. The involution $g$ is realized by the Bertini involution as described in Section \ref{sec:geiser-bertini}. By Lemma \ref{lem:geiser-bertini-irreducible}, $g$ is irreducible. 
\end{enumerate}
Therefore, there is a unique irreducible conjugacy class of order $2$ in $W_8$ and this class is realized by a Bertini involution. Corollary \ref{cor:involution-classification} then implies that this class is the only irreducible conjugacy class of order $2$ in $\Mod^+(M_8)$.
\end{proof}

We conclude by combining all of the lemmas above to prove Theorem \ref{thm:involutions-pt1}.
\begin{proof}[Proof of Theorem \ref{thm:involutions-pt1}]
Lemma \ref{lem:n-low} shows that each involution in $\Mod^+(M_n)$ of $1 \leq n \leq 4$ is reducible. Lemma \ref{lem:n6} and Propositions \ref{prop:n5}, \ref{prop:n7}, and \ref{prop:n8} show that the only irreducible involutions $g \in \Mod^+(M_n)$ for $5 \leq n \leq 8$ are those conjugate to the mapping classes of involutions on some $X = \Bl_P\CP^2$ induced by de Jonqui\'eres (of degree $d > 2$), Geiser, and Bertini involutions where $P$ is the set of its base points. Suppose $g \in \Mod^+(M_n)$ is realized by such an automorphism $\tilde g$ of $X$ via the diffeomorphism $\varphi: M_n \to X$. For any $f \in \Mod^+(M_n)$, there exists a diffeomorphism $F \in \Diff^+(M_n)$ with $[F] = f$ by Theorem \ref{thm:diffeo-realizable}. Hence $f^{-1}g f \in \Mod^+(M_n)$ is realized by $\tilde g$ via the diffeomorphism $\varphi\circ F: M_n \to X$. 
\end{proof}

Before considering the extension of Theorem \ref{thm:involutions-pt1} to Theorem \ref{thm:involutions}, we consider the notion of \emph{minimal pairs} considered by Bayle--Beauville (\cite{bayle--beauville}) in their classification of conjugacy classes of involutions in $\Cr(2)$. A pair $(S, \sigma)$ where $S$ is a rational surface and $\sigma$ is an involution of $S$ is called \emph{minimal} if any birational morphism $F: S \to S_0$ such that there exists an involution $\sigma_0$ of $S_0$ with $F \circ \sigma = \sigma_0 \circ F$ is an isomorphism. Corollary \ref{cor:minimal} is a reformulation of Theorem \ref{thm:involutions-pt1} using this language.

\begin{proof}[Proof of Corollary \ref{cor:minimal}]
Let $M$ be a del Pezzo manifold and let $g \in \Mod^+(M)$ be an irreducible mapping class of order $2$. By Theorem \ref{thm:involutions-pt1}, $g$ is realized by an involution $\sigma$ of a rational surface $S$ diffeomorphic to $M$. If $(S, \sigma)$ is not minimal then there exists some smooth rational curve $E \subseteq S$ such that $Q_M([E], [E]) = -1$ satisfying $\sigma(E) = E$ or $E \cap \sigma(E) = \emptyset$ by \cite[Lemma 1.1]{bayle--beauville}. In both cases, $M = M_n$ for some $1 \leq n \leq 8$. In the first case, $[\sigma]$ is reducible by Lemma \ref{lem:lattices}(\ref{lem:lattices2}). In the second case, $M = M_n$ for some $5 \leq n \leq 8$ by Theorem \ref{thm:involutions-pt1}. Note that $\Z\{[E], \sigma_*([E])\}$ is a $\Z$-submodule of $H_2(M; \Z)$ preserved by $[\sigma]$ to which the restriction of $Q_M$ is unimodular of signature $(0,2)$. Then $(\Z\{[E], \sigma_*([E])\}^\perp, Q_M)$ is preserved by $[\sigma]$ and is a unimodular lattice of signature $(1, n-2)$; it is isometric to $(H_2(M_{n-2}; \Z), Q_{M_{n-2}})$. Hence $[\sigma]$ is reducible. Therefore, $(S, \sigma)$ must be minimal if $[\sigma]$ is irreducible.

Now suppose $(S, \sigma)$ is a minimal pair where $S$ is a rational surface diffeomorphic to some del Pezzo manifold $M$ and $\sigma$ is an involution of $S$. All possible pairs $(S, \sigma)$ are listed in \cite[Theorem 1.4]{bayle--beauville}; we consider each case (i)-(vi) separately.
\begin{enumerate}
	\item[(i)] There exists a smooth $\CP^1$-fibration $f: S \to \CP^1$ and an involution $\tau$ of $\CP^1$ such that $f \circ \sigma = \tau \circ f$. Because $S$ is a geometrically ruled surface, $S$ is a $\CP^1$-bundle over $\CP^1$ by Noether--Enriques (\cite[Theorem III.4]{beauville}). Hence $S$ must be isomorphic to a Hirzebruch $\mathbb F_m$ for some $m \geq 0$ by \cite[Theorem 3.4.8]{gompf--stipsicz}. If $m > 0$ then any complex automorphism $\sigma$ of $S$ must preserve the unique irreducible curve $C$ of $S$ with self-intersection number $-m$ (given by a section of $f$) and $\sigma$ must also fix the homology class $[F]$ of the fiber $F$ of $f$. Because $[F]$ and $[C]$ span $H_2(S; \Z)$, this implies that $[\sigma] = \Id \in \Mod(M)$ so $[\sigma]$ does not have order $2$. If $m = 0$ then $S = \CP^1\times \CP^1$ is diffeomorphic to $M_*$. Any element of $\Mod(M_*)$ is irreducible by Lemma \ref{lem:n*0}.

	\item[(ii)] There exists a fibration $f: S \to \CP^1$ such that $f \circ \sigma = f$; the smooth fibers of $f$ are diffeomorphic to $\CP^1$ on which $\sigma$ induces a nontrivial involution and any singular fiber is the union of submanifolds diffeomorphic to $\CP^1$ exchanged by $\sigma$ meeting at one point. 

	Suppose $f$ has $s$-many singular fibers with $s > 0$. By the proof of \cite[Theorem 1.4]{bayle--beauville}, any singular fiber contains an exceptional divisor. Blowing down one of the components (call it $e_i$ for $1 \leq i \leq s$) in each singular fiber yields a geometrically ruled surface $f: S' \to \CP^1$, which is a $\CP^1$-bundle over $\CP^1$ by Noether--Enriques (\cite[Theorem III.4]{beauville}). This means that if $S_2\in H_2(S; \Z)$ is the class coming from a fiber of $f: S' \to \CP^1$ then 
	\[
		\sigma_*(S_2) = S_2, \qquad \sigma_*(e_i) = S_2-e_i
	\]
	for all $1 \leq i \leq s$. Because $H_2(S; \Z) = \Z\{S_1, S_2, e_1, \dots, e_{s}\}$ where $S_1$ is the class of a section of $f$, Lemmas \ref{lem:de-jon-determined} and \ref{lem:de-jonquieres-irreducible} imply that $[\sigma] \in \Mod(M_{s+1})$ is irreducible.

	If $s = 0$ then the same argument as in case (i) holds.

	\item[(iii), (iv)] The surface $S$ is isomorphic to $\CP^2$ or $\CP^1\times \CP^1$. Lemma \ref{lem:n*0} shows that $[\sigma] \in \Mod(M)$ is irreducible.
	\item[(v), (vi)] The surface $S$ is a del Pezzo surface of degree $2$ or $1$ and $f$ is the Geiser or Bertini involution respectively. Lemma \ref{lem:geiser-bertini-irreducible} shows that $[\sigma] \in \Mod(M)$ is irreducible. \qedhere
\end{enumerate} 
\end{proof}

\subsection{Extension to $\Mod(M_n)$ for $1 \leq n \leq 8$}\label{sec:extension}

In this section we compare the involutions $g \in \Mod(M_n)$ to involutions of $\Mod^+(M_n)$ to prove Theorem \ref{thm:involutions}. The following lemma will be used to construct involutions realizing irreducible order $2$ elements $g \in \Mod(M_n) - \Mod^+(M_n)$.
\begin{lem}\label{lem:tau-f-commute}
Let $P$ be a finite subset of $4 \leq n \leq 8$ points in general position contained in $\RP^2 \subseteq \CP^2$ and let $\tau_0: \CP^2 \to \CP^2$ be the map given by complex conjugation of the coordinates. Let $f: \Bl_P\CP^2 \to \Bl_P\CP^2$ be a complex automorphism of order $2$ induced by a birational map $f_0: \CP^2 \dashrightarrow \CP^2$ with base points given by $P$ and let $\tau: \Bl_P\CP^2 \to \Bl_P\CP^2$ be the map induced by $\tau_0$. Then $\tau$ and $f$ commute.
\end{lem}
\begin{proof}
For any polynomial $F \in \C[X,Y,Z]$, write $\overline F\in\C[X,Y,Z]$ to denote the polynomial obtained by conjugating the coefficients of $F$. There exist homogeneous polynomials $F, G, H \in \C[X,Y,Z]$ of degree $m$ such that 
\[
	f_0(q) = [F(q) : G(q) : H(q)]
\]
for all $q \notin P$. Then $g_0 := \tau_0 f_0 \tau_0$ is given by
\[
	g_0(q) = [\overline F(q) : \overline G(q) : \overline H(q)].
\]
If $q \in \CP^2$ such that $g_0$ is not defined at $q$ then
\[
	\tau_0(0) = \tau_0(\overline F(q)) = F(\tau_0(q)).
\]
Similarly, $G(\tau_0(q)) = H(\tau_0(q)) = 0$ and so $\tau_0(q) \in P$. Therefore, $q \in P$ because the points of $P$ are fixed by $\tau_0$. This shows that $g_0$ is birational and lifts to an automorphism $g$ of $\Bl_P\CP^2$. 

By construction, $g = \tau f \tau$ as a diffeomorphism of $\Bl_P\CP^2$. The action of $g_*$ on $H_2(\Bl_P\CP^2; \Z)$ coincides with the action of $f_*$ because $\tau_*$ acts by negation on $H_2(\Bl_P\CP^2; \Z)$. Therefore, $f = \tau f \tau$ because the homomorphism $\Aut(\Bl_P\CP^2) \to \Aut(H_2(\Bl_P\CP^2; \Z), Q_{\Bl_P\CP^2})$ is injective (\cite[Proposition 8.2.39]{dolgachev}).
\end{proof}

We finally extend Theorem \ref{thm:involutions-pt1} to prove Theorem \ref{thm:involutions}.
\begin{proof}[Proof of Theorem \ref{thm:involutions}]
Let $-I \in \Mod(M_n)$ denote the mapping class which acts by negation on $H_2(M_n; \Z)$, and let $-g = (-I) \circ g$ for any $g \in \Mod(M_n)$. If $g$ preserves some $\Z$-submodule $N \leq H_2(M_n; \Z)$ then $-g$ preserves $N$ as well. Therefore, $g$ is reducible if and only if $-g$ is reducible.

Let $g \in \Mod(M_n)$ be an irreducible element of order $2$. If $g \in \Mod^+(M_n)$ then Theorem \ref{thm:involutions-pt1} shows that $g$ is realized by a de Jonqui\'eres (of degree $d > 2$), Geiser, or Bertini involution. If $g \notin \Mod^+(M_n)$ then $-g \in \Mod^+(M_n)$. Theorem \ref{thm:involutions-pt1} shows that $-g$ is realized by de Jonqui\'eres (of degree $d > 2$), Geiser, or Bertini involutions. 
\begin{enumerate}
	\item If $-g$ is realized by Geiser or Bertini involutions then let $X = \Bl_P\CP^2$ where $P$ is a set of $n$ points in general position contained in $\RP^2 \subseteq \CP^2$. Let $f$ be the Geiser or Bertini involution of $X$ and let $\tau$ be the diffeomorphism of $X$ induced by complex conjugation on $\CP^2$. By Lemma \ref{lem:tau-f-commute}, $f\circ \tau$ has order $2$ in $\Diff^+(M_n)$. Then $[f\circ \tau] = g$ because $[\tau] = -I$ and $[f] = -g$.
	\item If $-g$ is realized by de Jonqui\'eres involutions then Lemma \ref{lem:de-jonquieres-cp2} shows that there exist $X = \Bl_P\CP^2$ where $P$ is a set of $n$ points in $\RP^2 \subseteq \CP^2$ and an automorphism $f \in \Aut(X)$ induced by a de Jonqui\'eres involution that commutes with the anti-biholomorphism $\tau$ induced by complex conjugation on $\CP^2$. Therefore, $f\circ\tau$ has order $2$ and $[f\circ\tau] = g$. \qedhere
\end{enumerate}
\end{proof}

\section{The smooth Nielsen realization problem for involutions}\label{sec:nielsen}
In this section we describe a construction that we call \emph{complex equivariant connected sums} and use it to prove the smooth Nielsen realization problem for involutions (Corollary \ref{cor:involutions}).
\subsection{Complex equivariant connected sums}\label{sec:cecs}
Finding representative diffeomorphisms of a mapping class $g$ of order two has distinct flavors depending on the irreducibility of $g$. We define \emph{complex equivariant connected sums} in order to realize order $2$ reducible mapping classes of del Pezzo manifolds. The definition here is specialized to $G = \Z/2\Z$ and is a special case of \emph{equivariant connected sums} which appear in \cite[(1.C)]{hambleton--tanase}. For a more general description, also see \cite[Section 2.2]{lee}.

Let $N_1, N_2$ be smooth manifolds and let $G = \Z/2\Z$. Fix a $G$-invariant Riemannian metric on both $N_1$ and $N_2$. Consider diffeomorphisms $g_i \in \Diff^+(N_i)$ of order two for $i = 1, 2$. Suppose there are points $p_i \in N_i$ for $i = 1, 2$ such that $p_i$ is fixed by $g_i$ and the tangent representations $G_i \to \SO(T_{p_i}N_i)$ are equivalent by an orientation-reversing isomorphism $\rho: T_{p_1}N_1 \to T_{p_2}N_2$. By the equivariant tubular neighborhood theorem (\cite[Theorem VI.2.2]{bredon}), there exist $G$-invariant neighborhoods of $p_i \in N_i$ for each $i = 1, 2$ which are $G$-equivariantly diffeomorphic to $T_{p_i}N_i$. We can now form as usual a connected sum $N_1 \# N_2$ by taking the $G$-equivariant neighborhoods of $p_1$ and $p_2$ in $N_1 - p_1$ and $N_2-p_2$ respectively and equivariantly identifying concentric annuli around $p_1$ and $p_2$ via the orientation-reversing map $\rho$. Then the connected sum $N_1 \# N_2$ has a natural smooth action of $G$. The $G$-manifold $(N_1 \# N_2, G)$ is called an \emph{equivariant connected sum}. See Figure \ref{fig:type1} for an illustration.
\begin{figure}
\centering
\includegraphics{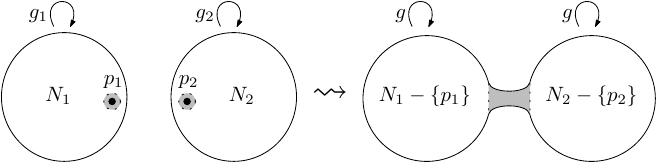}
\caption{The equivariant connected sum $(N_1\#N_2, G)$ where $G = \langle g \rangle \cong \Z/2\Z$.}\label{fig:type1}
\end{figure}

Consider $G \times N_2$. Suppose there exist points $p_1 \in N_1$ which is not fixed by $g_1$ and any $p_2 \in N_2$. Similarly as in the first case, the $G$-equivariant identification of the neighborhoods of the points in the $G$-orbit $\{p_1, g_1(p_1)\}$ of $p_1 \in N_1$ and the neighborhoods of the points $G \times \{p_2\}$ in $G\times N_2$ is denoted $(N_1 \# (G \times N_2), G)$ and is also called an \emph{equivariant connected sum}. See Figure \ref{fig:type2} for an illustration.

With these definitions in mind, we define a \emph{complex equivariant connected sum}. 
\begin{defn}\label{defn:cecs}
Let $M$ be a smooth, oriented manifold and let $G \cong \Z/2\Z \leq \Diff^+(M)$. The pair $(M, G)$ is called a \emph{complex equivariant connected sum} if one of the following holds:
\begin{enumerate}
	\item $(M, G)$ is $G$-equivariantly diffeomorphic to $(N, G)$ or $(\overline N, G)$ where $N$ is a complex manifold and $\overline N$ is the same manifold with the opposite orientation; each $g \in G \leq \Diff^+(N)$ is biholomorphic or anti-biholomorphic, \label{defn:cecs-1}
	\item $(M, G)$ is $G$-equivariantly diffeomorphic to an equivariant connected sum $(N_1 \# N_2, G)$ where $(N_1, G)$ and $(N_2, G)$ are complex equivariant connected sums, or
	\item $(M, G)$ is $G$-equivariantly diffeomorphic to an equivariant connected sum $(N_1 \# (\Z/2\Z) \times N_2, G)$ where $(N_1, G)$ is a complex equivariant connected sum.
\end{enumerate}
\end{defn}

If $G_0 \leq \Mod(M)$ is a finite group such that there exists a complex equivariant connected sum $(M, G)$ and $G \leq \Diff^+(M)$ is a lift of $G_0$ under the quotient $\pi: \Homeo^+(M) \to \Mod(M)$ then we say that $G_0$ is \emph{realizable by a complex equivariant connected sum}.

The following lemma is used in realizing reducible mapping classes of order $2$.
\begin{lem}\label{lem:connect-inv}
Let $M$ and $N$ be smooth $4$-manifolds and let $f_M \in \Diff^+(M)$ and $f_N \in \Diff^+(N)$ be diffeomorphisms of order $2$ fixing real surfaces $S_M \subseteq M$ and $S_N \subseteq N$ respectively. There is an equivariant connected sum $(M\#N, \langle f \rangle)$ where $\langle f \rangle \cong \Z/2\Z$ such that $f|_{M-p \subseteq M\#N} = f_M$ and $f|_{N-q\subseteq M\#N} = f_N$ with $p \in S_M$ and $q \in S_N$. Moreover, $f$ fixes a real surface in $M\#N$.
\end{lem}
\begin{proof}
Fix $f_M$- and $f_N$-invariant metrics on $M$ and $N$. For any $p \in S_M$, the action of $d(f_M)_p$ on $T_pM$ fixes $T_pS_M \subseteq T_pM$ and acts by negation on $T_pS_M^\perp \subseteq T_pM$. Similarly, the action of $d(f_N)_q$ on $T_qN$ fixes $T_qS_N \subseteq T_qN$ and acts by negation on $T_qS_N^\perp \subseteq T_qN$ for all $q \in S_N$. There is an orientation-reversing isomorphism $\varphi: T_qN \to T_pM$ taking $T_qS_N$ to $T_pS_M$ in an orientation-reversing way and taking $T_qS_N^\perp$ to $T_pS_M^\perp$ in an orientation-preserving way. By construction, $\langle f_M\rangle \to \SO(T_pM)$ and $\langle f_N \rangle \to \SO(T_qN)$ are equivalent by the orientation-reversing isomorphism $\varphi$ which forms the equivariant connected sum $(M\# N, \Z/2\Z)$. Moreover, $S_M \#S_N \subseteq M\#N$ is a real surface fixed by the resulting smooth $\Z/2\Z$-action.
\end{proof}

\subsection{The proof of Corollary \ref{cor:involutions}}\label{sec:nielsen-proof}
Let $M$ be a del Pezzo manifold. Throughout this section, we say that $\Id \in\Mod(M)$ is \emph{realizable by an order $2$ complex equivariant connected sum} if there exists a complex equivariant connected sum $(M, G)$ and $G \cong \Z/2\Z \leq \Diff^+(M)$ such that $G \leq \ker(\pi)$, where $\pi$ denotes the usual quotient $\Homeo^+(M) \to \Mod(M)$. Note that in this case, $G$ is not a lift of its image $\pi(G) = \Id \leq \Mod(M)$. 

The following lemma forms the base case of the inductive proof of Corollary \ref{cor:involutions}.
\begin{lem}\label{lem:nielsen-base}
Let $M = M_0$ or $M_*$. Any element $g \in \Mod(M)$ is realizable by order $2$ complex equivariant connected sum fixing a real surface.
\end{lem}
\begin{proof}
Note that $\Mod(M_0) = \Mod(\CP^2) \cong \{\pm \Id\}$, where the nontrivial element acts on $H_2(\CP^2; \Z)$ by negation. Then $[f_-] = -\Id \in \Mod(\CP^2)$ where $f_-$ is the involution $[X:Y:Z] \mapsto [\overline X: \overline Y: \overline Z]$ given by complex conjugation and fixes a real surface in $M_0$. Moreover, $[f_+] = \Id \in \Mod(\CP^2)$ where $f_+$ is the involution $[X:Y:Z] \mapsto [-X:Y:Z]$ which also fixes a real surface in $M_0$.

Note that $\Mod(M_*) = \langle c_1, c_2 \rangle \cong (\Z/2\Z)^2$ where
\[
	c_1 = \begin{pmatrix}
		-1 & 0 \\ 0 & -1
	\end{pmatrix}, \qquad c_2 = \begin{pmatrix}
	0 & 1 \\ 1 & 0
	\end{pmatrix}
\]
with respect to the $\Z$-basis $(S_1, S_2)$ of $H_2(M_*; \Z)$. Define 
\[
	f_{c_1}([X:Y], [W:Z]) = ([\overline X: \overline Y], [\overline W: \overline Z]), \qquad f_{c_2}([X:Y], [W:Z]) = ([W:Z], [X:Y]).
\]
The group $\langle f_{c_1}, f_{c_2} \rangle \leq \Diff^+(M_*)$ is a lift of $\Mod(M_*)$ under the quotient map $\pi: \Homeo^+(M_*) \to \Mod(M_*)$ with $\pi(f_{c_i}) = c_i$ for each $i = 1, 2$. It is straightforward to check that all nontrivial elements of $\langle f_{c_1}, f_{c_2} \rangle$ fix a real surface in $M_*$. The identity element is realized by $f_0: M_* \to M_*$ where $f_0([X:Y], [W:Z]) = ([-X:Y], [W:Z])$ which fixes a real surface in $M_*$. 
\end{proof}

The inductive step is handled by the lemma below. The proof is straightforward but included for the sake of completeness.
\begin{lem}\label{lem:nielsen-induction}
Fix $n \geq 1$. Suppose any $h \in \Mod(M_k)$ of order dividing $2$ is realizable by an order $2$ complex equivariant connected sum fixing a real surface for all $0 \leq k < n$ and $k = *$. If $g \in \Mod(M_n)$ is a reducible element of order dividing $2$, then $g$ is realizable by an order $2$ complex equivariant connected sum fixing a real surface. 
\end{lem}
\begin{proof}
Suppose $g$ is contained in the image of a standard inclusion 
\[
	\iota_*: \Mod(M_k) \times \Mod(\# \ell \overline{\CP^2}) \hookrightarrow \Mod(M_n)
\]
for some $k \leq n-1$ and $\ell = n-k$ or $k = *$ and $\ell = n-1$. Suppose $g = \iota_*(h_1, h_2)$ with $(h_1, h_2) \in \Mod(M_k) \times \Mod(\# \ell\overline{\CP^2})$. Up to conjugacy in $ \Mod(\#\ell\overline{\CP^2})$, any $h_2 \in \Mod(\#\ell\overline{\CP^2})$ of order dividing $2$ satisfies:
\[
	h_2: E_i \mapsto E_{j_i} \text{ or }\pm E_i
\]
for all $1 \leq i \leq \ell$ and some $j_i\neq i$. Because $h_2$ preserves $\Z\{E_i, E_{j_i}\}$ or $\Z\{E_i\}$, we may assume that $\ell = 2$ or $1$ respectively. Without loss of generality, suppose $h_2(E_1) = E_2$ or $h_2(E_1) = \pm E_1$.

Let $(M_{k}, \langle h \rangle)$ be an order $2$ complex equivariant connected sum (where $h$ is a diffeomorphism of $M_{k}$ fixing a real surface $S \subseteq M_{k}$) such that $[h] = h_1$. 
\begin{enumerate}
\item Suppose $\ell = 1$. If $h_2(E_1) = E_1$ then let $f: \overline{\CP^2} \to \overline{\CP^2}$ with $f: [X:Y:Z] \mapsto [-X:Y:Z]$. If $h_2(E_1) = -E_1$ then let $f: \overline{\CP^2} \to \overline{\CP^2}$ with $f: [X:Y:Z] \mapsto [\overline X:\overline Y:\overline Z]$. In either case, $f$ fixes a real surface in $\overline{\CP^2}$. There is a complex equivariant connected sum $(M_{k} \# \overline{\CP^2}, \Z/2\Z)$ fixing a real surface realizing $(h_1, h_2)$ by Lemma \ref{lem:connect-inv}.

\item Suppose $\ell = 2$ and $h_2(E_1) = E_2$. Then $(M_{k} \# ((\Z/2\Z)\times\overline{\CP^2}), \Z/2\Z)$ gives the desired complex equivariant connected sum. 
\end{enumerate}

Any reducible $g \in \Mod(M_n)$ of order dividing $2$ is conjugate to some $g_0 \in \Mod(M_n)$ contained in the image of a standard inclusion $\iota_*$; let $g = f^{-1}g_0f$ for some $f \in \Mod(M_n)$. By Theorem \ref{thm:diffeo-realizable}, there exists a diffeomorphism $F \in \Diff^+(M_n)$ with $[F] = f$. If $(M_n, G)$ is a complex equivariant connected sum realizing $g_0$ then $(M_n, F^{-1} G F)$ is a complex equivariant connected sum realizing $g$.
\end{proof}

With the inductive step in hand, we prove the smooth Nielsen realization problem for involutions on del Pezzo manifolds. 
\begin{proof}[Proof of Corollary \ref{cor:involutions}]
We will show that for any del Pezzo manifold $M$, any $g \in \Mod(M)$ of order dividing $2$ is realized by a complex equivariant connected sum of order $2$ fixing a real surface. 

The claim holds for $M = M_*$ and $M_0$ by Lemma \ref{lem:nielsen-base}. Fix $1 \leq n \leq 8$ and suppose that the claim holds for $M = M_k$ for all $0 \leq k < n$. We will prove the claim for $M = M_n$.

Let $g \in \Mod(M_n)$ be an element of order dividing $2$. If $g$ is reducible then $g$ is realized by a complex equivariant connected sum of order $2$ fixing a real surface by Lemma \ref{lem:nielsen-induction}. 

Suppose $g$ is irreducible. If $g \in \Mod^+(M)$ then Theorem \ref{thm:involutions} shows that $g$ is realized by a complex automorphism of some $\Bl_P\CP^2 \cong M$ induced by de Jonqui\'eres, Geiser, or Bertini involutions. All such automorphisms fix a complex curve in $\Bl_P\CP^2$. If $g \notin \Mod^+(M)$ then Theorem \ref{thm:involutions} shows that $g$ is realized by some anti-biholomorphism $f$ of order $2$ of a complex surface $\Bl_P\CP^2 \cong M$ and $-g$ is represented by an automorphism of $\Bl_P\CP^2$ induced by a de Jonqui\'eres, Geiser, or Bertini involution.

To show that $f$ fixes a real surface in $M$, we apply the Hirzebruch $G$-signature theorem (\cite[Section 9.2, (12)]{hirzebruch--zagier}) which says that if $f_0$ is a smooth involution of $M$ then
\[
	2\sigma(M/\langle f_0 \rangle) = \sigma(M) + \sum_C \defect_C
\]
where
\begin{enumerate}
	\item $\sigma(M/\langle f_0 \rangle)$ is the signature of the restriction of $Q_M$ to the fixed subspace of $H_2(M; \R)$ under $(f_0)_*$ (cf. \cite[Section 2.1, (22)]{hirzebruch--zagier}),
	\item $\sigma(M)$ is the signature of the $4$-manifold $M$, and
	\item the sum $\sum_C \defect_C$ is taken over the $2$-dimensional components of the fixed set of $f_0$ and $\defect_C$ denotes the quantity called the \emph{defect} of $C$. To be precise, the statement of the Hirzebruch $G$-signature theorem also involves defects $\defect_p$ associated to isolated fixed points $p$. However, $\defect_p = 0$ for all isolated fixed points $p$ when $f_0$ has order $2$. See \cite[Section 9.2]{hirzebruch--zagier} or \cite[Remark 4.4]{lee} for more details.
\end{enumerate}
We compute $2\sigma(M/\langle f \rangle)$ and $\sigma(M)$ in each of the three cases.
\begin{enumerate}
	\item Suppose $-g \in \Mod(M_n)$ is represented by a de Jonqui\'eres involution and $n = 5$ or $7$. By Lemma \ref{lem:de-jonquieres-homology}, the $\Z[\langle g \rangle]$-module structure of $H_2(M_n; \R)$ is isomorphic to
	\[
		H_2(M_n; \R) \cong C^{\oplus 2} \oplus \R^{\oplus n-1}
	\]
	where $C \cong \R$ as an $\R$-vector space and $g$ acts by negation. Moreover, this decomposition must be orthogonal and
	\[
		C^{\oplus 2} = \R\{S_2, \, 2S_1-e_1-\dots-e_{n-1}\}.
	\]
	With respect to this basis, the restriction of $Q_{M_n}$ to $C^{\oplus 2}$ is 
	\[
		Q_{M_n}|_{C^{\oplus 2}} = \begin{pmatrix}
		0 & 2 \\
		2 & -(n-1)
		\end{pmatrix}
	\]
	which has signature $0$. The signature of a direct sum of orthogonal subspaces is the sum of the respective signatures, meaning that 
	\[
		1 - n = \sigma(M_n) = \sigma(C^{\oplus 2}) + \sigma(H_2(M_n; \R)^{\langle g \rangle}) = \sigma(M_n/\langle f \rangle)
	\]
	The $G$-signature theorem implies that
	\[
		\sum_C \defect_C = 2 \sigma(M_n/\langle f \rangle) - \sigma(M_n) = 1-n \neq 0.
	\]
	Therefore, there exist real surfaces $C \subseteq M_n$ fixed by $f$.

	\item Suppose $-g \in \Mod(M_n)$ is represented by a Geiser or Bertini involution and $n = 7$ or $8$ respectively. There is an orthogonal decomposition
	\[
		H_2(M_n; \R) = \R\{K_{X_n}\} \oplus \mathbb E_n \otimes \R;
	\]
	here, $-g$ acts by negation on $\mathbb E_n$ and fixes $\R\{K_{X_n}\}$. Therefore, 
	\[
		H_2(M_n; \R)^{\langle g \rangle} = \mathbb E_n \otimes \R.
	\]
	The restriction of $Q_{M_n}$ to $\mathbb E_n$ is negative-definite so $\sigma(M_n / \langle f \rangle) = -n$. The $G$-signature theorem implies that
	\[
		\sum_C \defect_C = 2 \sigma(M_n/\langle f \rangle) - \sigma(M_n) = -n-1 \neq 0.
	\]
	Therefore, there exist real surfaces $C \subseteq M_n$ fixed by $f$.
\end{enumerate}
Therefore, any $g \in \Mod(M_n)$ of order dividing $2$ is realized by a complex equivariant connected sum of order $2$ fixing a real surface. The corollary now follows by induction on $n$.
\end{proof}

\bibliographystyle{alpha}
\bibliography{del-pezzo}

\bigskip
\noindent
Seraphina Eun Bi Lee \\
Department of Mathematics \\
University of Chicago \\
\href{mailto:seraphinalee@uchicago.edu}{seraphinalee@uchicago.edu}

\end{document}